\documentclass[12pt]{amsart}

\usepackage{amstext}
\usepackage{amsthm}
\usepackage{amsmath}
\usepackage{amssymb}
\usepackage{latexsym}
\usepackage{amsfonts}
\usepackage{graphicx}
\usepackage{texdraw}
\usepackage{graphpap}

\usepackage[plainpages=false, pagebackref=true, colorlinks, citecolor=blue, linkcolor=red]{hyperref}

\usepackage{amsrefs}

\input txdtools

\begin{document}

\newcommand{\ci}[1]{_{ {}_{\scriptstyle #1}}}

\newcommand{\norm}[1]{\ensuremath{\|#1\|}}
\newcommand{\abs}[1]{\ensuremath{\vert#1\vert}}
\newcommand{\p}{\ensuremath{\partial}}
\newcommand{\pr}{\mathcal{P}}

\newcommand{\pbar}{\ensuremath{\bar{\partial}}}
\newcommand{\db}{\overline\partial}
\newcommand{\D}{\mathbb{D}}
\newcommand{\B}{\mathbb{B}}
\newcommand{\Sp}{\mathbb{S}}
\newcommand{\T}{\mathbb{T}}
\newcommand{\R}{\mathbb{R}}
\newcommand{\Z}{\mathbb{Z}}
\newcommand{\C}{\mathbb{C}}
\newcommand{\N}{\mathbb{N}}
\newcommand{\scrH}{\mathcal{H}}
\newcommand{\scrL}{\mathcal{L}}
\newcommand{\td}{\widetilde\Delta}

\newcommand{\La}{\langle }
\newcommand{\Ra}{\rangle }
\newcommand{\rk}{\operatorname{rk}}
\newcommand{\card}{\operatorname{card}}
\newcommand{\ran}{\operatorname{Ran}}
\newcommand{\im}{\operatorname{Im}}
\newcommand{\re}{\operatorname{Re}}
\newcommand{\tr}{\operatorname{tr}}
\newcommand{\vf}{\varphi}
\newcommand{\f}[2]{\ensuremath{\frac{#1}{#2}}}


\newcommand{\entrylabel}[1]{\mbox{#1}\hfill}

\newenvironment{entry}
{\begin{list}{X}%
  {\renewcommand{\makelabel}{\entrylabel}%
      \setlength{\labelwidth}{55pt}%
      \setlength{\leftmargin}{\labelwidth}
      \addtolength{\leftmargin}{\labelsep}%
   }%
}%
{\end{list}}


\numberwithin{equation}{section}

\newtheorem{thm}{Theorem}[section]
\newtheorem{lm}[thm]{Lemma}
\newtheorem{cor}[thm]{Corollary}
\newtheorem{conj}[thm]{Conjecture}
\newtheorem{prob}[thm]{Problem}
\newtheorem{prop}[thm]{Proposition}
\newtheorem*{prop*}{Proposition}

\theoremstyle{remark}
\newtheorem{rem}[thm]{Remark}
\newtheorem*{rem*}{Remark}

\title{Stabilization in $H^\infty_\R(\D)$}
\author[B. D. Wick]{Brett D. Wick$^\ast$}
\address{Brett D. Wick, School of Mathematics\\ Georgia Institute of Technology\\ 686 Cherry Street\\  Atlanta, GA USA 30332-0160}
\email{wick@math.gatech.edu}
\thanks{$\ast$ Research supported in part by a National Science Foundation DMS Grant \# 0752703}


\begin{abstract}

In this paper we prove the following theorem:  Suppose that $f_1,f_2\in H^\infty_\R(\D)$, with $\norm{f_1}_\infty,\norm{f_2}_{\infty}\leq 1$, with
$$
\inf_{z\in\D}\left(\abs{f_1(z)}+\abs{f_2(z)}\right)=\delta>0.
$$
Assume for some $\epsilon>0$ and small, $f_1$ is positive on the set of $x\in(-1,1)$ where $\abs{f_2(x)}<\epsilon$ for some $\epsilon>0$ sufficiently small.  Then there exists $g_1, g_1^{-1}, g_2\in H^\infty_\R(\D)$ with 
$$
\norm{g_1}_\infty,\norm{g_2}_\infty,\norm{g_1^{-1}}_\infty\leq C(\delta,\epsilon)
$$
and
$$
f_1(z)g_1(z)+f_2(z)g_2(z)=1\quad\forall z\in\D.
$$
\end{abstract}

\maketitle
\setcounter{tocdepth}{1}
\tableofcontents
\section*{Notation}
\begin{entry}
\item[$:=$] equal by definition;\\
\item[$\C$] complex plane;\\
\item[$\re z$] real part of $z\in\C$;\\
\item[$\im z$] imaginary part of $z\in\C$;\\
\item[$\D$] the unit disc, $\D:=\{z\in\C:\abs{z}<1\}$;\\
\item[$\C_+$] the upper half-plane, $\C_+:=\{z\in\C: \im z>0\}$;\\
\item[$\T$] the unit circle, $\T:=\p\D$;\\
\item[$\R$] the real line, $\R:=\p\C_+$;\\
\item[$H^\infty$] the algebra of all bounded analytic functions on either $\C_+$ ($H^\infty(\C_+)$) or $\D$ ($H^\infty(\D)$);\\
\item[$H^\infty_\R$] the algebra of all bounded analytic functions with real Fourier coefficients on either $\C_+$ ($H^\infty_\R(\C_+)$) or $\D$ ($H^\infty_\R(\D)$);\\
\end{entry}

\setcounter{section}{-1}

\section{Introduction and Main Results}
\label{Intro.}
The stable rank of a ring (also called the Bass stable rank) was introduced by H. Bass in \cite{Bass} to assist in computations of algebraic K-Theory.  We recall the definition of the stable rank.

Let $A$ be an algebra with a unit $e$.  An n-tuple $a\in A^n$ is called \textit{unimodular} if there exists an $n$-tuple $b\in A^n$ such that $\sum_{j=1}^{n}a_jb_j=e$.  An $n$-tuple $a$ is called \textit{stable} or \textit{reducible} if there exists an $(n-1)$-tuple $x$ such that the $(n-1)$-tuple $(a_1+x_1a_n,\ldots, a_{n-1}+x_{n-1}a_{n})$ is unimodular.  The \textit{stable rank} (also called \textit{bsr}(A) in the literature) of the algebra $A$ is the least integer $n$ such that every unimodular $(n+1)$-tuple is reducible.  

The stable rank is a purely algebraic concept but can be combined with analysis when studying commutative Banach algebras of functions.  In this context, the stable rank is related to the zero sets of ideals, and the spectrum of the Banach algebra.  The stable rank for different algebras of analytic functions have been considered by many authors.  The computation of the stable rank of the disc algebra $A(\D)$ was shown to be one by Jones, Marshall and Wolff, \cite{JMW}.  The computation was done for sub-algebras of the disk algebra $A(\D)$ by Corach--Su\'arez, \cite{CS}, and Rupp \cite{R}.

For the Banach algebra $H^\infty(\D)$, the classification of its unimodular elements and its stable rank are well understood.  Carleson's Corona Theorem, see \cite{Carleson}, can be phrased as an $n$-tuple $(f_1,\ldots,f_n)\in H^\infty(\D)^n$ is unimodular if and only if it satisfies the Corona condition,
$$
\inf_{z\in\D}\left(\abs{f_1(z)}+\cdots+\abs{f_n(z)}\right)=\delta>0.
$$

The stable rank of $H^\infty(\D)$ was computed by S. Treil and is one of the motivations for this paper.  Treil's result is the following theorem:

\begin{thm}[S. Treil, \cite{TreilStable}]
\label{stableHinfty}
Let $f_1,f_2\in H^\infty(\D)$ be such that $\inf_{z\in\D}(\abs{f_1(z)}+\abs{f_2(z)})=\delta>0$.  Then there exists $g_1, g_2, g_1^{-1}\in H^\infty(\D)$ with $\norm{g_1}_\infty,\norm{g_2}_{\infty}$ and $\norm{g_1^{-1}}_\infty$ controlled by $C(\delta)$, a constant only depending on $\delta$, and  
$$
1=f_1(z)g_1(z)+f_2(z)g_2(z)\quad\forall z\in\D.
$$
\end{thm}

It is immediately apparent that Theorem \ref{stableHinfty} implies the stable rank of $H^\infty(\D)$ is one.  Questions about the stable rank of some sub-algebras of $H^\infty(\D)$ have been studied by Mortini, \cite{Mortini}.  

It is possible to phrase Treil's result \cite{TreilStable} in the language of Control Theory.  In this language, the result can be viewed as saying that it is possible to stabilize (\textit{in the sense given above}) a linear system (\textit{the Corona data, viewed as a rational function}) via a stable (\textit{analytic}) controller.  But, in applications of Control Theory, the linear systems and transfer functions  have real coefficients, so in this context Treil's result is physically meaningless.  From the point of view of Control Theory, it is important to know if results like Theorem \ref{stableHinfty} hold, but for a more physically meaningful algebra, and serves as the main motivation for this paper.  This paper is interested in questions related to the stable rank of a natural sub-algebra of $H^\infty(\D)$, the real Banach algebra $H^\infty_\R(\D)$.   In particular, does some variant of Theorem \ref{stableHinfty} hold for this algebra? 

First, recall that $H^\infty_\R(\D)$ is the subset of $H^\infty(\D)$ with the additional property of the following symmetry condition:
$$
f(z)=\overline{f(\overline{z})}\quad\forall z\in\D.
$$
An important property that we will use is that $f(z)\in\R$ when $z\in\D\cap\R$.  When we translate between $\D$ and $\C_+$ via the standard confromal map ($z\mapsto\frac{z+i}{z-i}$) this condition takes the following form,
$$
f(z)=\overline{f(-\overline{z})}\quad\forall z\in\C_+.
$$
Here we have that $f(z)\in\R$ when $z\in\C_+\cap i\R$.  These conditions are implying that the functions in $H^\infty_\R(\D)$ (or $H^\infty_\R(\C_+)$) possess a symmetry that is not present  for general $H^\infty(\D)$ (or $H^\infty(\C_+)$) functions.

Carleson's Corona result is inherited by the algebra $H^\infty_\R(\D)$.  More precisely, it is an immediate application of the usual Corona Theorem and the symmetry properties of $H^\infty_\R(\D)$ to show that an 
$n$-tuple $(f_1,\ldots,f_n)\in H^\infty_\R(\D)^n$ is unimodular if and only if it satisfies the Corona condition,
$$
\inf_{z\in\D}\left(\abs{f_1(z)}+\cdots+\abs{f_n(z)}\right)=\delta>0.
$$
Indeed, one direction is immediate, and in the other direction, if we know that
$$
\inf_{z\in\D}\left(\abs{f_1(z)}+\cdots+\abs{f_n(z)}\right)=\delta>0,
$$
then we can find a solution $(g_1,\ldots,g_n)\in H^\infty(\D)^n$.  We then symmetrize the $g_j$ via the operation
$$
\tilde{g}_j(z):=\f{g_j(z)+\overline{g_j(\overline{z})}}{2}.
$$
The $\tilde{g}_j\in H^\infty_\R(\D)$ and will then be the $H^\infty_\R(\D)$ Corona solution we are seeking.

This leads to the main question considered in this paper.  Is Theorem \ref{stableHinfty} true for the algebra $H^\infty_\R(\D)$? Namely, given Corona data $f_1$ and $f_2$ in $H^\infty_\R(\D)$, is there a solution $g_1$ and $g_2$ to the Corona problem with $g_1$ invertible in $H^\infty_\R(\D)$?  This can also be phrased as attempting to show that the stable rank of $H^\infty_\R(\D)$ is one.  

It is easy to see that there is an additional necessary condition.  Suppose that Theorem \ref{stableHinfty} were true for $H^\infty_\R(\D)$ functions, then we shall see that the real zeros of $f_1$ and $f_2$ must intertwine correctly.  Indeed, let $\lambda_1$ and $\lambda_2$ be real zeros of $f_2$.  Then we have
\begin{eqnarray*}
f_1(\lambda_1)g_1(\lambda_1) & = & 1\\
f_1(\lambda_2)g_1(\lambda_2) & = & 1.
\end{eqnarray*}
Now $f_1(\lambda_1)$ and $f_1(\lambda_2)$ must have the same sign at these zeros.  If this were not true, then without loss of generality, suppose that $f_1(\lambda_1)>0>f_1(\lambda_2)$.  Then $g_1(\lambda_1)>0>g_1(\lambda_2)$.  By continuity there will exist a point $\lambda_{12}$, between $\lambda_1$ and $\lambda_2$, with $g_1(\lambda_{12})=0$.  But this contradicts the fact that $g_1^{-1}\in H^\infty_\R(\D)$.  This necessary condition is not quite strong enough for our purposes, we instead need a slightly stronger condition.

First, we need to extend the definition of positive on real zeros.  Given a positive number $\epsilon>0$ we will instead will need that $f_1$ has the same sign on the set of $x\in(-1,1)$ where $\abs{f_2(x)}<\epsilon$.  Note that if $f_1$ has the same sign on this set then we will have that $f_1$ has the same sign on the real zeros of $f_2$.  Repeating and quantifying the argument given above one can see that a stronger necessary condition for a version of Theorem \ref{stableHinfty} to hold for $H^\infty_\R(\D)$ is that for some $\epsilon$ small enough, we have $f_1$ has the same sign on the set $\{x\in (-1,1):|f_2(x)|<\epsilon\}$.  Indeed, we have the following lemma.

\begin{lm}
Let $f_1,f_2\in H^\infty_\R(\D)$ and suppose that there exist $g_1, g_1^{-1}, g_2\in H^\infty_\R(\D)$ such that
$$
f_1g_1+f_2g_2=1\quad\forall z\in\D.
$$
Then, for $\epsilon>0$ sufficiently small we have that $f_1$ has the same sign on the set $\{x\in (-1,1):|f_2(x)|<\epsilon\}$.
\end{lm}

Note that if $f_1$ has the same sign on the set $\{x\in (-1,1):|f_2(x)|<\epsilon\}$, then by multiplying the function $f_1$ by $-1$ if necessary we can assume that $f_1$ is positive there.

We also observe an equivalent condition to $f_1$ being positive on the set $x\in(-1,1)$ such that $\abs{f_2(x)}<\epsilon$.  Note that the set
$$
\{x\in\D\cap\R: \abs{f_2(x)}<\epsilon\}=\bigcup_{l}^N Z_l
$$
where $\{Z_l\}_{l}^N$ is a finite collection of disjoint open intervals of $(-1,1)$.  The positivity condition on $f_1$ means that between any two consecutive open intervals $\{Z_l\}$ there are an \textit{even} number of open intervals $\{W_k\}$ such that $f_1$ changes sign, $f_1$ must have a real zero in the interval $W_k$ and there must be an even number of real zeros of $f_1$ between the real zeros of $f_2$.

With this definition, we now state the corrected form of the main Theorem from \cite{W}.  In the original paper \cite{W} the proof of the main result is unfortunately incorrect, however the method of proof is the correct approach.  The problem with the main result in \cite{W} is that the initial hypothesis was that $f_1$ was positive on the real zeros of $f_2$.  The condition that $f_1$ is positive on the zeros of $f_2$, while necessary is unfortunately not strong enough to be sufficient as can be seen by certain examples.  In this version of the paper, we provide the corrections necessary to obtain a true theorm.  The main result of this paper is the following theorem.  

\begin{thm}
\label{Hinfty_Corrected}
Suppose that $f_1,f_2\in H^\infty_\R(\D)$, $\norm{f_1}_\infty,\norm{f_2}_{\infty}\leq 1$.  Assume for some $\epsilon>0$ and small, $f_1$ has the same sign on the set $\{x\in (-1,1):|f_2(x)|<\epsilon\}$ and
$$
\inf_{z\in\D}\left(\abs{f_1(z)}+\abs{f_2(z)}\right)=\delta>0.
$$
Then there exists $g_1$, $g_1^{-1}$, $g_2\in H^\infty_\R(\D)$ with $\norm{g_1}_\infty,\norm{g_2}_\infty,\norm{g_1^{-1}}_\infty\leq C(\delta,\epsilon)$ and
$$
f_1(z)g_1(z)+f_2(z)g_2(z)=1\quad\forall z\in\D.
$$
\end{thm}
From now one, with out loss of generality, we will suppose that $f_1$ is positive on the set $\{x\in (-1,1):|f_2(x)|<\epsilon\}$.  
The original statement of the Theorem in \cite{W} did not take this parameter $\epsilon$ into consideration.

We remark that these results transfer immediately to analogous statements $H^\infty_\R(\C_+)$ via the standard conformal mapping between $\C_+$ and $\D$.  Throughout the paper, the adjective \textit{real symmetric} is used to indicate that the function in question satisfies the symmetry condition 
$$
f(z)=\overline{f(\overline{z})}\quad\forall z\in\D
$$
if the function is defined on the unit disc $\D$.  Or, 
$$
f(z)=\overline{f(-\overline{z})}\quad\forall z\in\C_+
$$
if the function is defined in $\C_+$.  This will always be clear from the context.  We also will work with either the upper half-plane or the disc, and will transfer the problem to either domain, depending upon where the problem is easiest to work with.  Finally, we remark that $C(\delta,\epsilon)$ will denote a constant depending on $\delta$ and $\epsilon$, which can change from line to line.\\

\noindent
\textbf{Acknowledgements}
The author thanks Raymond Mortini for providing interesting counterexamples to the original version  of Theorem which showed that, quite surprisingly, the constants $C(\delta)$ bounding the norms of the solutions  $(g_1,g_2)$ and the inverse $g_1^{-1}$ (in $H^\infty_\R(\D)$) to the Bezout equation $g_1f_1+g_2f_2 =1$ not only depends on $\delta$ as claimed, but also must depend on a stronger condition about the positivity, namely the parameter $\epsilon$.  The author also thanks Kalle Mikkola for a similar observation.  Finally, the author thanks Sergei Treil for numerous helpful discussions.

\section{Idea of the Proof}
\label{ideaofproof}

The method of proof is inspired by Treil's proof in \cite{TreilStable}, but must be suitably modified.  A key component of the modifications is to exploit the symmetry properties of $H^\infty_\R(\D)$ functions.  Additionally, it is important to include the condition about $f_1$ being positive on the set $\{x\in (-1,1):|f_2(x)|<\epsilon\}$ in an appropriate way.

It is straightforward to demonstrate that it is enough to prove Theorem \ref{Hinfty_Corrected} only in the case of real symmetric rational functions whose zeros satisfy the additional condition about positivity on certain zeros.  While not immediately clear, it is also true that one can show that it is sufficient to prove Theorem \ref{Hinfty_Corrected} in the case of real symmetric finite Blaschke products possessing this condition as well.  This can be seen from the appropriate modifications of the original proof by Carleson and the proof by Treil, \cite{TreilStable}.  Thus, we specialize to the situation where we have real symmetric simple Blaschke products with the positivity condition on real zeros.  

To prove Theorem \ref{Hinfty_Corrected} we begin by solving an interpolation problem.  We now work in the upper half plane where the problem is easier to explain.  Suppose that $f_1$ and $f_2$ are real symmetric finite simple Blaschke products, which satisfy the condition that $f_1$ is positive on the set $\{y\in\R :|f_2(iy)|<\epsilon\}$ and 
$$
\inf_{z\in\C_+}\left(\abs{f_1(z)}+\abs{f_2(z)}\right)=\delta>0.
$$
Note that the function $f_1(z)$ satisfies $\abs{f_1(z)}\geq\f{\delta}{2}$ on the set $\{z\in\C_+:\abs{f_2(z)}<\f{\delta}{2}\}$.  Further, by the maximum principle, observe that each component of this set is simply connected.  Since $f_1$ is rational, we have that $\log f_1$ is bounded on this set, and moreover, because  $f_1$ is positive on the set $\{y\in \R:|f_2(iy)|<\epsilon\}$ and the Corona condition holds there is a well defined branch of $\log f_1$ with respect to the symmetric set $\{z\in\C_+:\abs{f_2(z)}<\delta'\}$, where $\delta'$ is sufficiently small compared with $\delta$ and $\epsilon$, that can be chosen with the additional property that $\log f_1(z)=\overline{\log f_1(-\overline{z})}$.

With this in hand, we now use the following Theorem about interpolation of functions.  

\begin{thm}
\label{InterpolateLogReal_corrected}
Let $B$ be a real symmetric Blaschke product with simple zeros, and let $\sigma$ denote its zero set.  Given $0<\gamma<1$, let $\vf$ be a real symmetric analytic function on the set $\{z:\abs{B(z)}<\gamma\}$ and satisfying $\abs{\vf(z)}\leq 1$ there.  Then there exists a real symmetric function $h\in H_\R^\infty(\C_+)$ such that
$$
\vf(z)=h(z)\quad\forall z\in\sigma.
$$
Moreover, $\norm{h}_\infty\leq C(\gamma)\norm{\varphi}_\infty$.
\end{thm}

We remark that it suffices to have $\varphi$ symmetric only on the zeros of the function $B$, however, we state the result in a slightly stronger form.  
To prove Theorem \ref{InterpolateLogReal_corrected}, one simply takes the resulting function that exists in $H^\infty(\C_+)$ (as given by Carleson in \cite{Carleson}), call it $l(z)$, and then symmetrize it by setting $h(z)=\frac{l(z)+\overline{l(-\overline{z})}}{2}$.  Since $l$ does the interpolation and everything is symmetric, the result then follows.

We apply this Theorem with $B=f_2$ and $\varphi=\log f_1$ and $\delta'$ a small number compared to both $\delta$ and $\epsilon$.  Since $f_1$ is rational, we have a bounded branch of the logarithm $\log f_1$ on the set 
$$
\{z:\abs{f_2(z)}<\delta'\}.
$$
Additionally, by choosing $\delta'=\delta'(\epsilon)>0$ small enough we have $f_1$ is positive on the set $\{y\in\R :|f_2(iy)|<\epsilon\}$ so we can interpolate the logarithm of $f_1$ with a function $h\in H^\infty_\R(\C_+)$ with $\norm{h}_\infty\leq C(\delta,\epsilon)\norm{\log f_1}_\infty$, and
$$
e^{h(z)}=f_1(z)\quad\textnormal{for all $z$ in the zero set of }f_2.
$$
The function $e^h$ is invertible in $H^\infty_\R(\C_+)$ and there is a function $G\in H^\infty_\R(\C_+)$ with $e^h=f_1+f_2G$.  

Unfortunately, this is not enough to conclude the proof of the theorem.  If $\log f_1$ were bounded on $\{z\in \C_+:\abs{f_2(z)}<\delta'\}$ by a constant only depending on $\delta$ and $\epsilon$ and \textit{not} on the degrees of $f_1$ and $f_2$, we would be done.  However, this is not generally true, so we need a method to overcome this difficulty.  To do this we will find an analytic function $\kappa$ that is real symmetric and  will ``correct'' the function $f_1$.  To correct the function $f_1$ it suffices to prove the following proposition, which has been corrected from \cite{W}.

To find the correcting function, we will prove the following propositions.  The point where the paper \cite{W} was incorrect in was the analogue of Proposition \ref{Correcting}, which was not strong enough to conclude that there was a bounded logarithm in $H^\infty_\R(\D)$. In that version of the proposition it was only assumed that $f_1$ is positive on the real zeros of $f_2$, i.e., setting $\epsilon=0$.  With this change the argument is now correct.

\begin{prop}
\label{Correcting}
Let $p,q\in H^\infty_\R(\C_+)$ be finite simple real symmetric Blaschke products with $\inf_{z\in\C_+}(\abs{p(z)}+\abs{q(z)})=\delta>0$ such that for some $\epsilon>0$ $p$ has the same sign on the set $\{y\in \R:|q(iy)|<\epsilon\}$.  Then there exists an analytic function $\kappa$ with the following properties:
\begin{itemize}
\item[(i)] $\abs{\re \kappa(z)}\leq C(\delta,\epsilon)\quad\forall z\in\C_+$;\\
\item[(ii)] $\abs{\log p(z)-\kappa(z)}\leq C(\delta,\epsilon)$ for all $z$ in $\{z\in\C_+:\abs{q(z)}<\delta'\}$ for some $0<\delta'$ sufficient small with respect to $\delta$ and $\epsilon$ and an appropriate branch of $\log p$ on the set $\{z\in\C_+:\abs{q(z)}<\delta'\}$;\\
\item[(iii)] $\kappa(z)=\overline{\kappa(-\overline{z})}\quad\forall z\in\C_+$.
\end{itemize}
\end{prop}

To find $\kappa$ we will construct an auxiliary  function $V$.
\begin{prop}
\label{CorrectingV_new}
Let $p,q\in H^\infty_\R(\C_+)$ be finite simple real symmetric Blaschke products with $\inf_{z\in\C_+}(\abs{p(z)}+\abs{q(z)})=\delta>0$ such that for some $\epsilon>0$ $p$ has the same sign on the set $\{y\in \R:|q(iy)|<\epsilon\}$.  Then there exists a function $V$ with the following properties:
\begin{itemize}
\item[(i)] $\abs{\re V(z)}\leq C(\delta,\epsilon)\quad\forall z\in\C_+$;\\
\item[(ii)] $\abs{\log p(z)-V(z)}\leq C(\delta,\epsilon)$ for all $z$ in $\{z\in\C_+:\abs{q(z)}<\delta'\}$ for some $0<\delta'$ sufficient small with respect to $\delta$ and $\epsilon$ and an appropriate branch of $\log p$ on the set $\{z\in\C_+:\abs{q(z)}<\delta'\}$;\\
\item[(iii)] $V(z)=\overline{V(-\overline{z})} \quad\forall z\in\C_+$;\\
\item[(iv)] some conditions to guarantee the existence of a bounded solution $v$ on the entire upper half-plane $\C_+$ of the equation $\pbar v=\pbar V$, in particular:\\
\begin{itemize}
\item[(a)] $\abs{\Delta V(z)} \im z\,dxdy$ is a Carleson measure with intensity $C(\delta,\epsilon)$;\\
\item[(b)] $\abs{\p V(z)}\,dxdy$ is a Carleson measure with intensity $C(\delta,\epsilon)$;\\
\item[(c)] $\abs{\Delta V(z)}\leq\f{C(\delta,\epsilon)}{(\im z)^2}\quad\forall z\in\C_+$.
\end{itemize}
\end{itemize}
\end{prop}
\noindent
For simplicity, we will assume that $p$ is positive on the set $\{y\in\R:|q(iy)|<\epsilon\}$.  We can always reduce to this case by multiplying by $-1$.

We remark that it is always possible to construct a symmetric function $V$ from the problem.  This simply follows from the symmetrization of the proof in Treil \cite{TreilStable} and that the data $p$ and $q$ has this symmetry property.  However, we remark that the positivity condition on $p$ and $q$ is necessary to  to construct the symmetric branch of logarithm this correcting function $V$ will approximate.  It was here that the author made a mistake in the previous version of \cite{W}.

Proposition \ref{CorrectingV_new} immediately implies Proposition \ref{Correcting}.  To see this, once we have constructed $V$,  set $\kappa=V-v$.  Trivially, we have that $\kappa$ is analytic because the $\pbar$-derivatives of $V$ and $v$ agree.  In Section \ref{dbarsection} it will be shown that the solution $v$ to $\overline{\partial} v=\overline{\partial} V$ will also possess the property $v(z)=\overline{v(-\overline{z})}$ since the data $V$ is symmetric.  So $\kappa$ will also have the symmetry property as well.  Condition (i) on $\kappa$ then follows from the corresponding condition on $V$ and the boundedness of the solution $v$ (condition (iv) above).  Finally, we have 
$$
\abs{\log p -\kappa}=\abs{\log p-V+v}\leq\abs{\log p-V}+\abs{v}.
$$
So, the boundedness of $v$ and condition (ii) on $V$ imply the corresponding condition on $\kappa$.  We will prove Proposition \ref{CorrectingV_new} in Sections \ref{mainconstruction} and \ref{Visgood}.

Now, consider the function $e^{-\kappa}f_1$.   Conditions (i) and (iii) of Proposition \ref{Correcting} imply that $e^{\kappa}\in H^\infty_\R(\C_+)$.  Note that $e^{-\kappa}f_1$ will be positive on the set $\{y\in \R:|f_2(iy)|<\epsilon\}$.  Also, observe that Condition (ii) implies that $e^{-\kappa}f_1$ has a bounded branch of logarithm that is real symmetric on 
$$
\left\{z\in\C_+:\abs{f_2(z)}<\delta'\right\}.
$$  
Applying Theorem \ref{InterpolateLogReal_corrected} and the argument that followed, we obtain $e^h=f_1e^{-\kappa}+f_2G
_1$ with $h,G_1\in H^\infty_\R(\C_+)$ with $\norm{h}_\infty\leq C(\delta,\epsilon)$.  Set $g_2:=G_1e^{-h}$ and $g_1:=e^{-(\kappa+h)}$.  Then we have that $g_1, g_2, g_1^{-1}\in H^\infty_\R(\C_+)$ such that $\norm{g_1}_\infty$, $\norm{g_2}_\infty$, $\norm{g_1^{-1}}_\infty$ is controlled by $C(\delta,\epsilon)$ and 
$$
f_1(z)g_1(z)+f_2(z)g_2(z)=1\quad\forall z\in\C_+.
$$

This argument then shows that to prove Theorem \ref{Hinfty_Corrected}, we need to establish Proposition \ref{CorrectingV_new}.

\section{\texorpdfstring{Construction of Bounded Solutions to the $\pbar$-Equation with $H_{\R}^\infty(\D)$ Data}{Construction of Bounded Solutions to the d-bar equation with real symmetric data}}
\label{dbarsection}
As is well known, solutions to the $\pbar$-equation on the disc are intimately connected with solutions to the Corona Problem because of connections between $\pbar$-equations and Carleson measures.  We now recall the definition of Carleson measures.  Let $I$ be an interval in $\R$ and form the Carleson square $Q=Q(I)$ over $I$,
$$
Q(I):=\left\{z\in\C_+:\re z\in I,\ \im z\leq\abs{I}\right\}.
$$
Then we say a non-negative measure $\mu$ in the upper half-plane $\C_+$ is  a Carleson measure if
$$
\sup_{I}\f{\mu(Q(I))}{\abs{I}}:=K<\infty,
$$
with the supremum taken over all intervals $I$ in $\R$.  The constant $K$ will be called the \textit{intensity} of the Carleson measure.  It is immediate to transfer these notions to the disc $\D$.  We have the following well known theorem, which can be found in \cite{TreilStable}.
\begin{thm}
\label{dbaronD}
Let $V$ be a $\mathcal{C}^2$ function on the unit disc $\D$ which is continuous up to the boundary $\T$.  Suppose that 
\begin{enumerate}
\item $\abs{\pbar V(z)}\,dxdy$ is a Carleson measure with intensity $K_1$;\\
\item $\abs{\Delta V(z)}(1-\abs{z}^2)\,dxdy$ is a Carleson measure with intensity $K_2$;\\
\item $\abs{\Delta V(z)}\leq\f{K_3}{(1-\abs{z}^2)^2}$.
\end{enumerate}
Then the equation 
$$
\pbar v=\pbar V
$$
has a bounded solution $v$ on all of $\D$ (not only the boundary $\T$) with
$$
\abs{v(z)}\leq C(K_1,K_2,K_3)\quad\forall z\in\D.
$$
\end{thm}

Now we want to show that if the function $V$ has the property that $V(z)=\overline{V(\overline{z})}$ for all $z\in\D$, then this property is inherited by the solution $v$.  This leads to the following theorem.

\begin{thm}
\label{realdbaronD}
Let $V$ be a $\mathcal{C}^2$ function on the unit disc $\D$ which is continuous up to the boundary $\T$.  Suppose that 
\begin{enumerate}
\item $V(z)=\overline{V(\overline{z})}$ for all $z\in\D$;\\
\item $\abs{\pbar V(z)}\,dxdy$ is a Carleson measure with intensity $K_1$;\\
\item $\abs{\Delta V(z)}(1-\abs{z}^2)\,dxdy$ is a Carleson measure with intensity $K_2$;\\
\item $\abs{\Delta V(z)}\leq\f{K_3}{(1-\abs{z}^2)^2}$.
\end{enumerate}
Then the equation 
$$
\pbar v=\pbar V
$$
has a bounded solution $v$ on all of $\D$ (not only the boundary $\T$) with
$$
\abs{v(z)}\leq C(K_1,K_2,K_3)\ \textnormal{and}\ v(z)=\overline{v(\overline{z})}\quad\forall z\in\D.
$$
\end{thm}

\begin{proof}[Proof of \ref{realdbaronD}]
We first apply Theorem \ref{dbaronD} to find a solution $v$ which is bounded on all of $\D$.  Then we replace it with the following function
$$
\tilde{v}(z):=\f{v(z)+\overline{v(\overline{z})}}{2}.
$$
Note that we have $\tilde{v}(z)=\overline{\tilde{v}(\overline{z})}$ and that $\norm{\tilde{v}}_{H^\infty(\D)}\leq\norm{v}_{H^\infty(\D)}\leq C(K_1,K_2,K_3)$.  We only need that $\pbar\tilde v=\pbar V$ for all $z\in\D$.  But, this follows from direct application of the chain rule.  Indeed, set $f^{\dagger}(z):=\overline{f(\overline{z})}$ it suffices to show that
$$
\pbar f^{\dagger}(z)=(\pbar f)^{\dagger}(z).
$$
Once we have this, then
\begin{eqnarray*}
\pbar V(z) & = & \pbar\tilde{v}(z)
\end{eqnarray*}
follows immediately.  Now, by definitions and the chain rule we have

\begin{eqnarray*}
\pbar f^{\dagger}(z) & = & \f{\p}{\p\overline{z}} \overline{f(\overline{z})}\\
 & = & \overline{\f{\p}{\p z} f(\overline{z})}\\
 & = & \overline{\f{\p f}{\p z}(\overline{z})\f{\p \overline{z}}{\p z}+\f{\p f}{\p \overline{z}}(\overline{z})\f{\p z}{\p z}}\\
 & = & \overline{\f{\p f}{\p \overline{z}}(\overline{z})}=(\pbar f)^{\dagger}(z).
\end{eqnarray*}
\end{proof}

Using the conformal equivalence between $\D$ and $\C_+$, it is possible to translate the above theorem, leading to the following.

\begin{thm}
\label{realdbaronH}
Let $V$ be a $\mathcal{C}^2$ function on the upper half-plane $\C_+$ which is continuous up to the boundary $\R$ and at the point $z=\infty$.  Suppose further that 
\begin{enumerate}
\item $V(z)=\overline{V(-\overline{z})}$ for all $z\in\C_+$;\\
\item $\abs{\pbar V(z)}\,dxdy$ is a Carleson measure with intensity $K_1$;\\
\item $\abs{\Delta V(z)}\im z\,dxdy$ is a Carleson measure with intensity $K_2$;\\
\item $\abs{\Delta V(z)}\leq\f{K_3}{(\im z)^2}$.
\end{enumerate}
Then the equation 
$$
\pbar v=\pbar V
$$
has a bounded solution $v$ on all of $\C_+$ (not only the boundary $\R$) with
$$
\abs{v(z)}\leq C(K_1,K_2,K_3)\ \textnormal{and}\ v(z)=\overline{v(-\overline{z})}\quad\forall z\in\C_+.
$$
\end{thm}

\noindent
These theorems will be used to find $H_\R^\infty(\D)$ solutions to certain $\pbar$-equations.

\section{Main Construction}
\label{mainconstruction}
\noindent
Recall that it only remains to prove Proposition \ref{CorrectingV_new}.\\

\noindent
\textbf{Proposition \ref{CorrectingV_new}}
\textit{Let $p,q\in H^\infty_\R(\C_+)$ be finite simple real symmetric Blaschke products with $\inf_{z\in\C_+}(\abs{p(z)}+\abs{q(z)})=\delta>0$ such that for some $\epsilon>0$ $p$ is positive on the set $\{y\in\R :|q(iy)|<\epsilon\}$.  Then there exists a function $V$ with the following properties:
\begin{itemize}
\item[(i)] $\abs{\re V(z)}\leq C(\delta,\epsilon)\quad\forall z\in\C_+$;\\
\item[(ii)] $\abs{\log p(z)-V(z)}\leq C(\delta,\epsilon)$ for all $z$ in $\{z\in\C_+:\abs{q(z)}<\delta'\}$ for some $0<\delta'$ sufficient small with respect to $\delta$ and $\epsilon$ and an appropriate branch of $\log p$ on the set $\{z\in\C_+:\abs{q(z)}<\delta'\}$;\\
\item[(iii)] $V(z)=\overline{V(-\overline{z})} \quad\forall z\in\C_+$;\\
\item[(iv)] some conditions to guarantee the existence of a bounded solution $v$ on the entire upper half-plane $\C_+$ of the equation $\pbar v=\pbar V$, in particular:\\
\begin{itemize}
\item[(a)] $\abs{\Delta V(z)} \im z\,dxdy$ is a Carleson measure with intensity $C(\delta,\epsilon)$;\\
\item[(b)] $\abs{\p V(z)}\,dxdy$ is a Carleson measure with intensity $C(\delta,\epsilon)$;\\
\item[(c)] $\abs{\Delta V(z)}\leq\f{C(\delta,\epsilon)}{(\im z)^2}\quad\forall z\in\C_+$.
\end{itemize}
\end{itemize}
}

The construction of $V$ is inspired by the construction given by S. Treil in \cite{TreilStable}, however we need to appropriately modify it to take advantage of the symmetry that $H^\infty_\R(\C_+)$ functions possess.  Parts of the construction below could be given by simply taking the construction of Treil in \cite{TreilStable} and reflecting in the imaginary axis.  However, in the interest of making the paper as self contained as possible, instead of just pointing to the construction in \cite{TreilStable} we instead will describe the necessary details of the construction.

The main approach to this Proposition is the construction of a symmetric Carleson contour.  We use the modification developed by Bourgain in \cite{Bourgain} and exploited by Treil in \cite{TreilStable}.  We further modify the method to force symmetry into the Carleson regions, which is possible since we are working with the algebra $H^\infty_\R(\C_+)$.  This is an essential point in the argument.

We let $b_a(z)$ denote the elementary Blaschke factor in $H^\infty(\C_+)$ with zero at $a\in\C_+$, i.e., $b_a(z):=\f{z-a}{z-\overline{a}}$.  The following lemmas will be of use.  

\begin{lm}
\label{BlasEst}
Let $B=\prod_{a\in\sigma}b_a$ be a finite Blaschke product with simple zeros.  Suppose that for a given $z\in\C_+$ and $\gamma>0$ we have
$$
\abs{b_a(z)}\geq\gamma\quad\forall a\in\sigma.
$$
Then
$$
\sum_{a\in\sigma}\f{2\im z \im a}{\abs{z-\overline{a}}^2}\leq\log\f{1}{\abs{B(z)}}\leq\f{1}{\gamma}\sum_{a\in\sigma}\f{2\im z\im a}{\abs{z-\overline{a}}^2}.
$$
\end{lm}
The next lemma will be used to construct the Carleson regions appropriately adapted to our functions.

\begin{lm}
\label{CarlRegion}
Let $B$ be a finite Blaschke product with simple zeros with $\sigma$ denoting its zero set.  Let $Q=Q(I)$ be a square with the base $I$ and suppose that there is a point $z_0$ in the top half of $Q$ with $\abs{B(z_0)}\geq\eta>0$.  Then, given $M<\infty$, there exists a collection of disjoint closed subinterval $\{I_k\}$ of $I$ with the following properties:
\begin{enumerate}
\item[(i)] $\sum\abs{I_k}\leq 20\log\f{1}{\eta}M^{-1}\abs{I}$;\\
\item[(ii)] $\sum_{a\in\sigma\cap Q(3I_k)}\im a\geq M\abs{I_k}\quad\forall k$;\\
\item[(iii)] If $z\in Q\setminus\cup_kQ(I_k)$, then $\sum_{a\in\sigma}\f{\im z \im a}{\abs{z-\overline{a}}^2}\leq C(M+\log\f{1}{\eta})$ with $C$ an absolute constant;\\
\item[(iv)] The measure $\mu:=\sum_{a\in\sigma\cap Q\setminus\cup_kQ(I_k)}\im a\  \delta_a$ is a Carleson measure with intensity at most $5M$.
\end{enumerate}
\end{lm}

The proof of this lemma is a stopping time argument. See Bourgain \cite{Bourgain} for a version of this lemma or Treil \cite{TreilStable} for the version indicated above.  Since the functions we have possess additional symmetry, we will apply the above lemmas to ``half'' of our function.  This is a key difference between the result found in \cite{TreilStable}.  With these Lemmas, we now construct generations of closed intervals and regions in the following manner.  First, note that for functions in $H^\infty_\R(\C_+)$ we have the following symmetry property
$$
f(z)=\overline{f(-\overline{z})}\quad\forall z\in\C_+.
$$
Recall that a function is \textit{real symmetric} if it satisfies this symmetry condition.  Note that this symmetry is interchanging the left and right halves of $\C_+$.  What is important in our case is that for finite Blaschke products with this symmetry property, a point $a$ is a zero if and only if $-\overline{a}$ is a zero.  We will use this symmetry in the selection of generations of intervals.  Let 
$$
\sigma_R:=\{a\in\sigma :\re a\geq 0\}\quad\sigma_L:=\{a\in\sigma :\re a< 0\}.
$$
This splits the zero set $\sigma$ of $p$ into its left and right halves.  Choose an interval such that the real part of all zeros of the function are contained in this interval.  Since we are working with $H^\infty_\R(\C_+)$ functions, then the zeros will be symmetric and it will be possible to choose a symmetric interval, i.e. $I=(-L,L)$ for some $L$.  Now, take a square $Q=Q(I)$ where $(I=[0,L))$, which contains the zero set $\sigma_R$ and we have $\abs{p(z_0)}\geq\delta'$ for some point $z_0$ in the top half of $Q$ and $\abs{p(z)}\geq\delta'$ for all $z\notin Q(I)\cup Q(-I)$.  One should observe that a rectangle is all that is required to contain the zeros of the function.

Choose $M=M(\delta)>2\cdot100\log\f{1}{\delta'}$, where $\delta'$ is smaller than both $\delta$ and $\epsilon$.  We apply the discussion from the previous paragraph and Lemma \ref{CarlRegion} with the choice $M$, $\eta=\delta'$ and $B=p$ and $Q=Q(I)\supset\sigma_R$.  We thus obtain a sequence of disjoint closed sub-intervals of $I$, $\{I_k\}$, such that:
\begin{enumerate}
\item[(i)] $\sum\abs{I_k}\leq\f{1}{2\cdot 5}\abs{I}$;\\
\item[(ii)] $\sum_{a\in\sigma\cap Q(3I_k)}\im a\geq 2\cdot 100\log\f{1}{\delta'}\abs{I_k}\quad\forall k$;\\
\item[(iii)] If $z\in Q\setminus\cup_kQ(I_k)$, then $\sum_{a\in\sigma}\f{\im z \im a}{\abs{z-\overline{a}}^2}\leq C\log\f{1}{\delta'}$ with $C$ an absolute constant;\\
\item[(iv)] The measure $\mu:=\sum_{a\in\sigma\cap Q\setminus\cup_kQ(I_k)}\im a\,  \delta_a$ is a Carleson measure with intensity at most $5M\geq250\log\f{1}{\delta'}$.
\end{enumerate}

Given this collection of intervals, we now form a new collection of intervals in the following manner.  It is here that we exploit the symmetry of $H^\infty_\R(\C_+)$ functions.  This is a point in the proof where the construction of Treil must be modified, and the symmetry allows this.  Set $\mathcal{K}_1:=\{I_k,-I_k\}=\{I_k'\}$.  Then these intervals are closed disjoint sub-intervals (after a possible union of two of them) of an interval of the form $(-L,L)$.  These intervals also possess the property that $I_k', -I_k'\in\mathcal{K}_1$ (again there could be the possibility that $I_k'=-I_k'$ for some $k$).  Set $\tilde{Q}:=Q(I)\cup Q(-I)$.  Furthermore, they have the property that:
\begin{enumerate}
\item[(i)] $\sum\abs{I_k'}\leq \f{1}{5}\abs{I}$;\\
\item[(ii)] $\sum_{a\in\sigma\cap Q(3I_k')}\im a\geq 2\cdot 100\log\f{1}{\delta'}\abs{I_k'}\quad\forall k$;\\
\item[(iii)] If $z\in\tilde{Q}\setminus\cup_kQ(I_k')$, then $\sum_{a\in\sigma}\f{\im z \im a}{\abs{z-\overline{a}}^2}\leq C\log\f{1}{\delta'}$ with $C$ an absolute constant;\\
\item[(iv)] The measure $\mu:=\sum_{a\in\sigma\cap\tilde{Q}\setminus\cup_kQ(I_k')}\im a\, \delta_a$ is a Carleson measure with intensity at most $10M\geq 500\log\f{1}{\delta'}$.
\end{enumerate}

This is a straightforward application of the symmetry that $H^\infty_\R(\C_+)$ functions possess.  We indicate some of this now.  Since Property (i) holds for the collection $\{I_k\}$, by disjointness and symmetry it will hold for $\{-I_k\}$ and hence for $\mathcal{K}_1$.   Property (ii) also holds by symmetry.  Since $\mathcal{K}_1=\{I_k'\}=\{I_k,-I_k\}$, and we know that Property (ii) holds for the collection $\{I_k\}$ by reflection in the y-axis and the symmetry of the zero set of $H^\infty_\R(\C_+)$ functions, Property (ii) holds for $\{-I_k\}$ as well.  Property (iii) and Property (iv) also follow immediately by the symmetry of $H^\infty_\R(\C_+)$ functions

Note that for any interval $J\in\mathcal{K}_1$, we have $\abs{p(z)}<\delta'$ for any $z$ in the top half of $Q(J)$.  This follows by Lemma \ref{BlasEst} and the above construction.  We have
$$
\log\f{1}{\abs{p(z)}}\geq\sum_{a\in\sigma}\f{2\im a\im z}{\abs{z-\overline{a}}^2}\geq\sum_{a\in\sigma\cap Q(3J)}\f{2\im a\im z}{\abs{z-\overline{a}}^2}.
$$
But, for $z$ in the top half of $Q(J)$ and $z\in Q(3J)$, we have $\im z\geq\f{\abs{J}}{2}$ and $\abs{z-\overline{a}}\leq 2\sqrt{2}\abs{J}$.  So by the construction of the intervals in $\mathcal{K}_1$ and the properties that they possess, we have
$$
\log\f{1}{\abs{p(z)}}\geq\sum_{a\in\sigma\cap Q(3J)}\f{1}{8\abs{J}}\im a\geq\f{1}{8}M>\log\f{1}{\delta'}.
$$

We iterate the above construction of construct generations of intervals and corresponding Carleson regions.  Fix an interval $J\in\mathcal{K}_1$ and let $\mathcal{D}(J)$ be the maximal dyadic sub-intervals $J'\subset J$ such that the top half of each $Q(J')$ contains a point $z_0$ where $\abs{p(z_0)}>\delta'$.  Note that, by the symmetry of the function $p$, we will obtain a symmetric selection of intervals.  Since $p$ is a finite Blaschke product, then $\mathcal{D}(J)$ is finite as well.  Moreover, 
$$
J=\bigcup_{J'\in\mathcal{D}(J)}J'.
$$
For each $J\in\mathcal{K}_1$, we set $\mathcal{U}(J):=\textnormal{clos}\left(Q(J)\setminus\cup_{J'\in\mathcal{D}(J)}Q(J')\right)$ and we set $\mathcal{R}_1:=\{\mathcal{U}(J):J\in\mathcal{K}_1\}$.  The set $\mathcal{R}_1$ is the first generation of Carleson regions.  We should note that by symmetry, since $J, -J\in\mathcal{K}_1$, then $\mathcal{U}(J)$ and $\mathcal{U}(-J)$ are symmetric with respect to reflection in the imaginary axis.  Also, it is easy to see that the sets $\mathcal{D}(J)$ and $\mathcal{D}(-J)$ will be symmetric in this manner as well.

For each $J'\in\mathcal{D}(J)$ with $J\in\mathcal{K}_1$, we apply the above construction to obtain a second generation of intervals $\mathcal{K}_2$.  Note that we only need to perform the construction for half the intervals; the other half is obtained by symmetry, i.e., reflection in the imaginary axis.  For each $J\in\mathcal{K}_2$, we form $\mathcal{U}(J)$ and $\mathcal{R}_2:=\{\mathcal{U}(J):J\in\mathcal{K}_2\}$.

Finally, we define $\mathfrak{U}:=\cup_{j\geq 1}\cup_{J\in\mathcal{K}_j}\mathcal{U}(J)$ and $\mathfrak{R}$ as the collection of all connected components of $\mathfrak{U}$.  We also set $\sigma_1=\sigma\setminus\mathfrak{U}$.  Since we have preserved symmetry throughout the construction, we will have the property that $J,-J\in\mathcal{K}_j$.  Also,  $\mathcal{U}(J)$ and $\mathcal{U}(-J)$ will be symmetric with respect to reflection in the imaginary axis.  This implies that the set $\sigma_1$ will also be symmetric, i.e., $a\in\sigma_1$ if and only if $-\overline{a}\in\sigma_1$.  Additionally, by construction, we have the property that if $z\in\mathfrak{U}$ then $\abs{p(z)}\leq\delta'$.

Letting $l_{\partial\mathfrak{U}}$ denote the arc length on the boundary of $\partial\mathfrak{U}$ of the region $\mathfrak{U}$, and letting $\delta_a$ denote the unit mass at the point $a$, the following result is straightforward.

\begin{prop}
\label{CarlesonMeasures}
Let $\mathfrak{U}$ and $\sigma_1$ be as above.  Then
\begin{enumerate}
\item[(i)] The measure $l_{\partial\mathfrak{U}}$ is a Carleson measure with intensity at most $C(\delta)$;\\
\item[(ii)] The measure $\sum_{a\in\sigma_1}\im a\ \delta_a$ is a Carleson measure with intensity at most $C(\delta)$.
\end{enumerate}
\end{prop}

Using the regions $\mathfrak{U}$ and $\sigma_1$, we will construct the function $V$.  The function $V$ is constructed as a finite sum of summands of two types, with each type having two sub-types.  Here again, the construction of Treil must be appropriately modified.  This is necessary because we need to make sure that the function $V$ possesses the symmetry property $V(z)=\overline{V(-\overline{z})}$ for all $z\in\C_+$.

We have to now split the construction of the function $V$ into the cases when the zeros of $p$ are off the imaginary axis and when the zeros are on the imaginary axis.  This is necessary so that we can define the appropriate branch of the logarithm.  When the zeros are off the axis, the construction is straightforward and one can just take the construction in Treil \cite{TreilStable} and symmetrize it appropriately.  However, when the zeros are on the imaginary axis, we need a different construction, different than what appears in \cite{W}.  This is a place where the original paper did not have the appropriate construction.

If the region in question had an even number of real zeros, then we would be fine.  In this case we have that above and below the region, a certain Blaschke product will be positive, and so we can select a branch of the logarithm so that it is real symmetric.  The difficulty in the construction when the zeros are on the imaginary axis is that they could occur in ``groups'' with an odd multiplicity.  If this happens, then it will be impossible to define the branch of the logarithm that is real symmetric, since above the region we are interested a related Blaschke factor will be positive while below the region it will be negative, and so we can not define a branch of the logarithm that is real symmetric.  This difficulty will \textit{not} allow us to define the approximating function $V$ in a simple manner, and instead we will have to ``pair'' certain zeros to overcome this difficulty.  The idea will be to pair regions with an odd number of real zeros to make regions with an even number of real zeros so that we can then appeal to the construction when there are an even number of real zeros.  

Recall that
\begin{equation}
\label{finiteUnion}
\left\{iy:|q(iy)|<\delta'\right\}=\bigcup_{l=1}^N Z_l
\end{equation}
is a finite union of disjoint open intervals $Z_l$.  Note that in this case, since we are supposing that $p$ is positive on the set $\{y\in\R :|q(iy)|<\epsilon\}$ then, between any two consecutive intervals $Z_l$ we will have an \textit{even} number of disjoint open intervals, $\{W_k\}$ where $p$ must change sign.  In particular, we see that between any two consecutive intervals $Z_{l_1}$ and $Z_{l_2}$ the function $p$ must have an even (or none) number of real zeros.  

We must split the construction now into the collection of zeros that are either on the imaginary axis, or those off the imaginary axis.  The two constructions will be similar, though we have to augment the construction when the zeros are on the imaginary axis.  The construction when the zeros are off the imaginary axis is as appears in \cite{W}, while the construction for the zeros on the imaginary axis is new.

\subsection{Zeros off the Imaginary Axis}

\subsubsection{Summands of the First Type: Zeros off the Imaginary Axis}

In this subsection, we do the construction of the summands corresponding to the zeros in $\sigma_1$.  Recall that $b_a(z):=\f{z-a}{z-\overline{a}}$ is the simple Blaschke factor with zero at the point $a\in\C_+$.  One immediately observes that $b_a\in H^\infty_\R(\C_+)$ if and only if $a=-\overline{a}$, namely $a=iy$ lies on the imaginary axis.  We also comment that $\overline{b_a(-\overline{z})}=b_{-\overline{a}}(z)$. This can be applied to a product of terms, and one sees $b_ab_{-\overline{a}}\in H^\infty_\R(\C_+)$.  These will be the elementary building blocks used.

We first deal with the case where when the zeros of $p$ are off the imaginary axis.  Let $D_a$ denote the disc with center at the point $a$ of radius $\delta'\im a$, and let $T_a=\p D_a$.    Let $I_a$ be the vertical slit which connects the circle $T_a$ with the real axis at the point $\re a$, i.e., 
$$
I_a=\{z\in\C_+:\re z=\re a, 0\leq \im z\leq(1-\delta')\im a\}.
$$  
By construction, we have that the symmetric point $-\overline{a}\in\sigma_1$.  We also construct the corresponding disc $D_{-\overline{a}}$ and $T_{-\overline{a}}$ and the corresponding slit $I_{-\overline{a}}$.  The above construction is explained in Figure 1.

\begin{figure}[htbp]
\label{Circle}
\begin{center}
\begin{picture}(400,200)
\put(0,0){\includegraphics{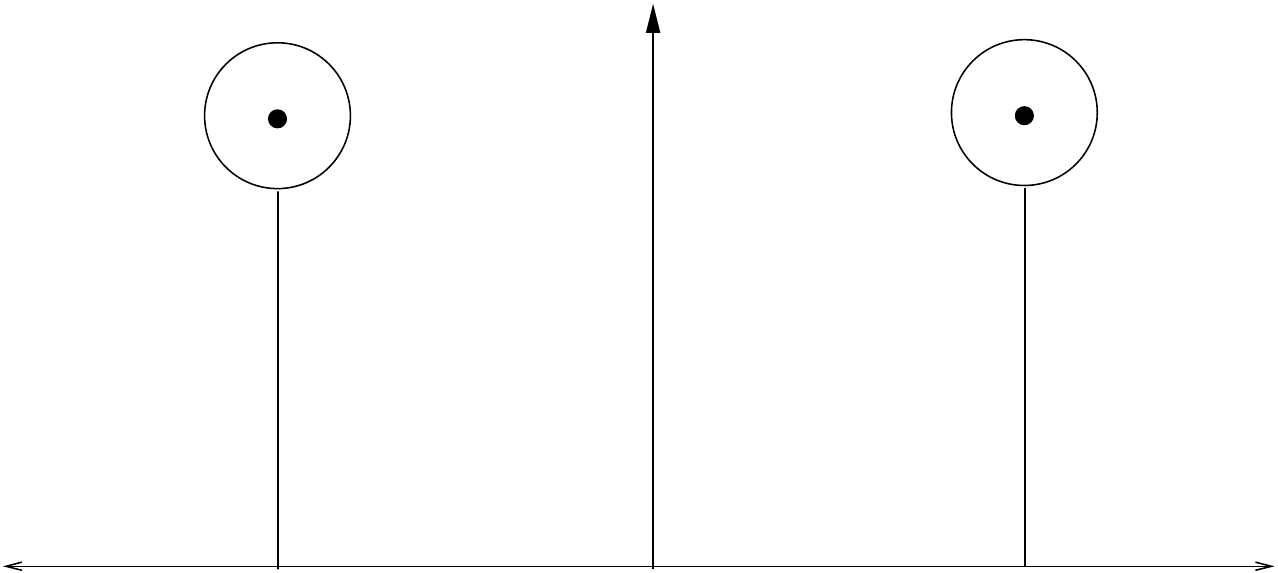}}
\put(375,10){$\mathbb{R}$}
\put(190,170){$i\mathbb{\R}$}
\put(295,120){$a$}
\put(75,120){$-\overline{a}$}
\put(60,80){$I_{-\overline{a}}$}
\put(280,80){$I_a$}
\put(320,150){$T_a$}
\put(100,150){$T_{-\overline{a}}$}
\end{picture}
\end{center}
\caption{}
\end{figure}

Then we have that $\abs{b_a b_{-\overline{a}}}\leq C(\delta,\epsilon)$ on the set $D_a\cup D_{-\overline{a}}$ and $\abs{b_a b_{-\overline{a}}}\geq C(\delta,\epsilon)$ on the compliment in $\C_+$.  We then define a function $\vf:\C_+\setminus(T_a\cup I_a\cup T_{-\overline{a}}\cup I_{-\overline{a}})$ in the following manner
\begin{displaymath}
\vf(z):=\left\{
\begin{array}{lll} 
0 & : & z\in D_a\cup D_{-\overline{a}}\\
\log (b_a b_{-\overline{a}}) & : & \textnormal{otherwise}.
\end{array}\right. 
\end{displaymath}
Here we use the principal branch of the logarithm.  Note that the term $b_ab_{-\overline{a}}\in H^\infty_\R(\C_+)$, and moreover it is positive on the imaginary axis $i\R$, which allows us to define the branch of the logarithm so that it has imaginary part zero.  Thus, we have that $\vf$ is bounded analytic function, with bounded depending on $\delta$ and $\epsilon$, that is real symmetric on the set $\C_+\setminus(T_a\cup I_a\cup T_{-\overline{a}}\cup I_{-\overline{a}})$.    We change $\vf$ in the $\f{\delta'}{100}\im a$--neighborhood of $T_a\cup I_a\cup T_{-\overline{a}}\cup I_{-\overline{a}}$ to obtain a smooth function $V_a$ on $\C_+$ such that 
\begin{itemize}
\item[(i)] $\abs{\pbar V_a(z)}\leq\f{C(\delta,\epsilon)}{\im a}$;\\
\item[(ii)] $\abs{\Delta V_a(z)}\leq\f{C(\delta,\epsilon)}{(\im a)^2}$;\\
\item[(iii)] $V_a(z)=\vf(z)$ if $\textnormal{dist}\left(z, T_a\cup I_a\cup T_{-\overline{a}}\cup I_{-\overline{a}}\right)>\f{\delta'}{100}\im a $;\\
\item[(iv)] $V_a(z)=\overline{V_a(-\overline{z})}$.
\end{itemize}

The function $V_a$ is obtained by the convolution of $\vf$ with a smooth kernel possessing the symmetry property.  Properties (i) and (ii) follow from well known estimates for $H^\infty(\C_+)$ functions.  Property (iii) and (iv) are a simple construction and verification that, if you convolve with a function that possesses the real symmetry property then the function $V_a$ will possess this property as well.  The construction of the mollifying function is straightforward.

We remark that with this construction, when the zeros off the imaginary axis get close to the axis these discs $D_a$ will end up overlapping.  This doesn't present a problem since there is only finite overlap which makes the constants slightly worse.  

One could instead replace the construction with the following one that avoids the finite overlap.  If we used the above construction, then in certain instances, the two circles would intersect, so instead we take one circle about the two points.  The construction is almost identical then.  We let $D_{a,-\overline{a}}$ denote the disc with center $\im a i$ and radius $\delta'\im a$.  When the points are close enough to the imaginary axis $i\R$ the disc $D_{a,-\overline{a}}$ contains the points $a$ and $-\overline{a}$.  We let $T_{a,-\overline{a}}:=\partial D_{a,-\overline{a}}$ denote the boundary of the circle.  We let $I_{a,-\overline{a}}$ denote the vertical slit which connects the boundary $T_{a,-\overline{a}}$ with the real line.  Namely, $I_{a,-\overline{a}}=\{z\in\C_+: \re z=0, 0\leq \im z\leq (1-\delta')\im a\}$.  The construction is explained in Figure 2.

\begin{figure}[htbp]
\label{ellipse}
\begin{center}
\begin{picture}(400,200)
\put(0,0){\includegraphics{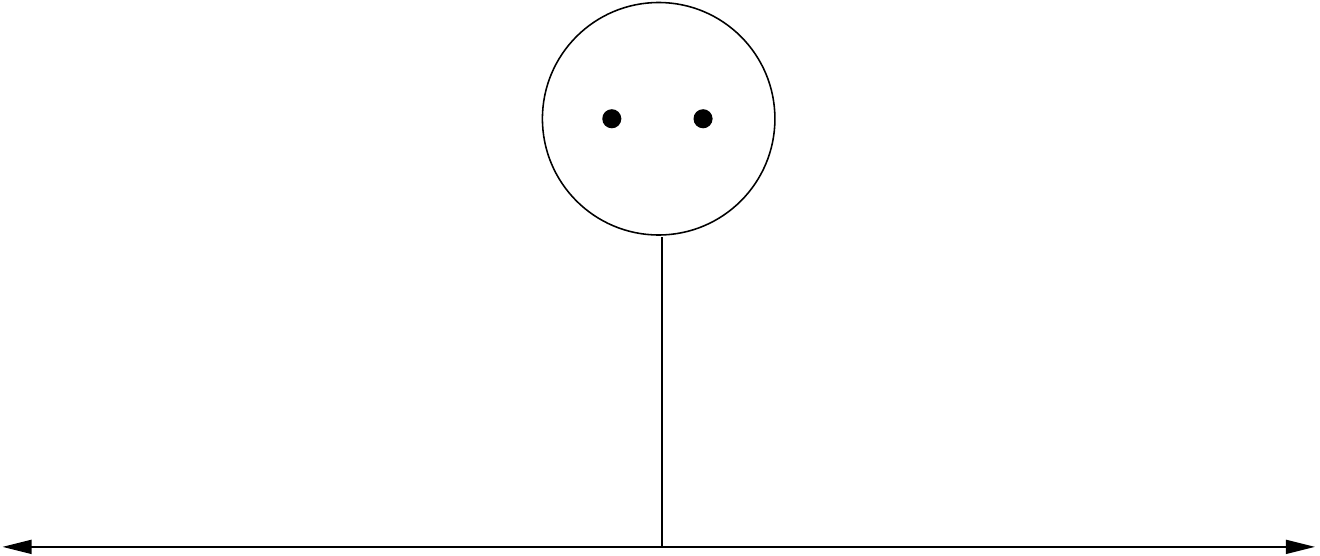}}
\put(375,10){$\mathbb{R}$}
\put(208,120){$a$}
\put(158,120){$-\overline{a}$}
\put(200,50){$I_{a,-\overline{a}}$}
\put(200,170){$T_{a,-\overline{a}}$}
\end{picture}
\end{center}
\caption{}
\end{figure}

Then we will have that $\abs{b_a b_{-\overline{a}}}\leq C(\delta,\epsilon)$ on $D_{a,-\overline{a}}$ and $\abs{b_a b_{-\overline{a}}}\geq C(\delta,\epsilon)$ for points in the compliment.  We then define $\vf$ in the following manner,
\begin{displaymath}
\vf(z):=\left\{
\begin{array}{lll} 
0 & : & z\in D_{a,-\overline{a}}\\
\log (b_a b_{-\overline{a}}) & : & \textnormal{otherwise}.
\end{array}\right. 
\end{displaymath}
Clearly we have that $\varphi$ is analytic, bounded (with a bound depending on $\delta$ and $\epsilon$) and real symmetric because of a similar argument in the case when there are two discs.  We again smooth $\vf$ to find our function $V$.  Namely, we change $\vf$ in a $\f{\delta'}{100}\im a$ neighborhood of $T_{a,-\overline{a}}\cup I_{a,-\overline{a}}$ to obtain a smooth function on $\C_+$ such that:
\begin{itemize}
\item[(i)] $\abs{\pbar V_a(z)}\leq\f{C(\delta,\epsilon)}{\im a}$;\\
\item[(ii)] $\abs{\Delta V_a(z)}\leq\f{C(\delta,\epsilon)}{(\im a)^2}$;\\
\item[(iii)] $V_a(z)=\vf(z)$ if $\textnormal{dist}\left(z, T_{a,-\overline{a}}\cup I_{a,-\overline{a}}\right)>\f{\delta'}{100}\im a$;\\
\item[(iv)] $V_a(z)=\overline{V_a(-\overline{z})}$.
\end{itemize}

\subsubsection{Summands of the Second Type:  Zeros off the Imaginary Axis}

The construction of these summands is slightly more involved.  Let $\mathcal{R}$ be a connected component of $\mathfrak{U}$.  Because the construction of the Carleson contour was done in a symmetric manner, we have two possibilities, either $\mathcal{R}=-\overline{\mathcal{R}}$ or $\mathcal{R}\neq -\overline{\mathcal{R}}$.  In the second case, due to the symmetry of the construction, we must have had that there is a related region that is symmetric with $\mathcal{R}$.  We will pair these together.

Let 
$$
\mathcal{R}_{\delta'}:=\left\{z\in\C_+:\inf_{a\in\mathcal{R}}\abs{b_a(z)}<\f{\delta'}{100}\right\}.
$$
Set $B_\mathcal{R}:=\prod_{a\in\sigma\cap\mathcal{R}}b_a$.  By Lemma \ref{CarlRegion} for any $z\in\partial\mathcal{R}$ we have
$$
\sum_{a\in\sigma\cap\mathcal{R}}\f{\im a\im z}{\abs{z-\overline{a}}^2}\leq C(\delta,\epsilon).
$$
It is straightforward to demonstrate that for any $z\in\p\mathcal{R}_{\delta'}$, we have an identical estimate only with a larger constant.  So, by Lemma \ref{BlasEst} we have $\abs{B_\mathcal{R}(z)}\geq\gamma(\delta,\epsilon)$, which by the maximum modulus principle holds for all $z\notin\p\mathcal{R}_{\delta'}$.

If $\mathcal{R}\neq-\overline{\mathcal{R}}$, then the domains $\mathcal{R}$ contain no real zeros by the construction.  Note, that $B_\mathcal{R}$ is not in $H^\infty_\R(\C_+)$, but $B_\mathcal{R}B_{-\overline{\mathcal{R}}}$ is.  We define 
\begin{displaymath}
\vf(z):=\left\{
\begin{array}{lll} 
0 & : & z\in \mathcal{R}_{\delta'}\cup-\overline{\mathcal{R}}_{\delta'}\\
\log (B_\mathcal{R}B_{-\overline{\mathcal{R}}}) & : & \C_+\setminus(\mathcal{R}_{\delta'}\cup-\overline{\mathcal{R}}_{\delta'}),
\end{array}\right.
\end{displaymath}
for some branch of the logarithm.  We also split the set $\C_+\setminus(\mathcal{R}_{\delta'}\cup-\overline{\mathcal{R}}_{\delta'})$ into connected components and technically define $\vf$ as an appropriate branch of logarithm on each such component.  Again, the function $\vf$ is real symmetric since we can chose a branch of the logarithm that has imaginary part being zero because we have that .  We arrive at a smooth function by convolving $\vf$ with an appropriate real symmetric function.  Then let $V_\mathcal{R}$ be the mollification of the function $\vf$ constructed using information from the region $\mathcal{R}$.  

The splitting is obtained in a similar manner to what appears in Treil's construction.  We recall the splitting that is used in \cite{TreilStable}.  Each component $\mathcal{R}$ is a union of Carleson regions $\mathcal{U}$.  Each region was constructed as 
$$
\mathcal{U}(I)=\textnormal{clos}\left(Q(I)\setminus\bigcup_{J\in\mathcal{D}(I)}Q(J)\right)
$$
where $\mathcal{D}(I)$ was a family of dyadic subintervals of $I$ with $I=\cup_{J\in\mathcal{D}(I)}J$.  For each such dyadic sub-interval $J$ with center $c$ at the point $c$, we draw a vertical interval (slit) $[c,c+i\abs{J}]$.  In addition to these vertical slits, we also need $\Gamma$-slits.  To construct the $\Gamma$-slits, consider a vertical sub-interval of $\p\mathcal{R}$ which is maximal with respect to inclusion.  Let $I=[a+ib,a+ic]$ $a\in\R$ $c>b>0$.  For any integer $k$, $k\geq 1$, we do nothing if $c2^{-k-2}\leq b$.  If $c2^{-k-2}>b$ we draw in $\C_+\setminus\mathcal{R}$ a horizontal interval of length $2\cdot 2^{-k-1}$ with the endpoint $a+ic2^{-k}$.  Then we draw the vertical interval which connects the other endpoint to the real axis.  Note that we can do this construction for one region $\mathcal{R}$ and then symmetry will deal with the corresponding region on the other side of $i\R$.

We need the following proposition from Treil's paper.
\begin{prop}
All slits (vertical and $\Gamma$-slits) corresponding to a component $\mathcal{R}$ are disjoint and the origin of each slit is the only point of its intersection with the component of $\mathcal{R}$.  Moreover, if we consider for each slit $S$ of altitude $d$, its $\f{\delta'}{100}$--neighborhood $S_{\delta'}$ (here we take the usual, not hyperbolic neighborhood), all $S_{\delta'}$ are also disjoint.
\end{prop}
The proof of this Proposition is a direct repeat of what appears in Treil's work \cite{TreilStable}, so we omit it.  We simply remark that we have forced additional symmetry upon our Carleson regions so that we can arrive at functions with certain symmetry properties.  The figure below explains the general construction.

\begin{figure}[htbp]
\label{region1}
\begin{center}
\begin{picture}(400,200)
\put(0,0){\includegraphics{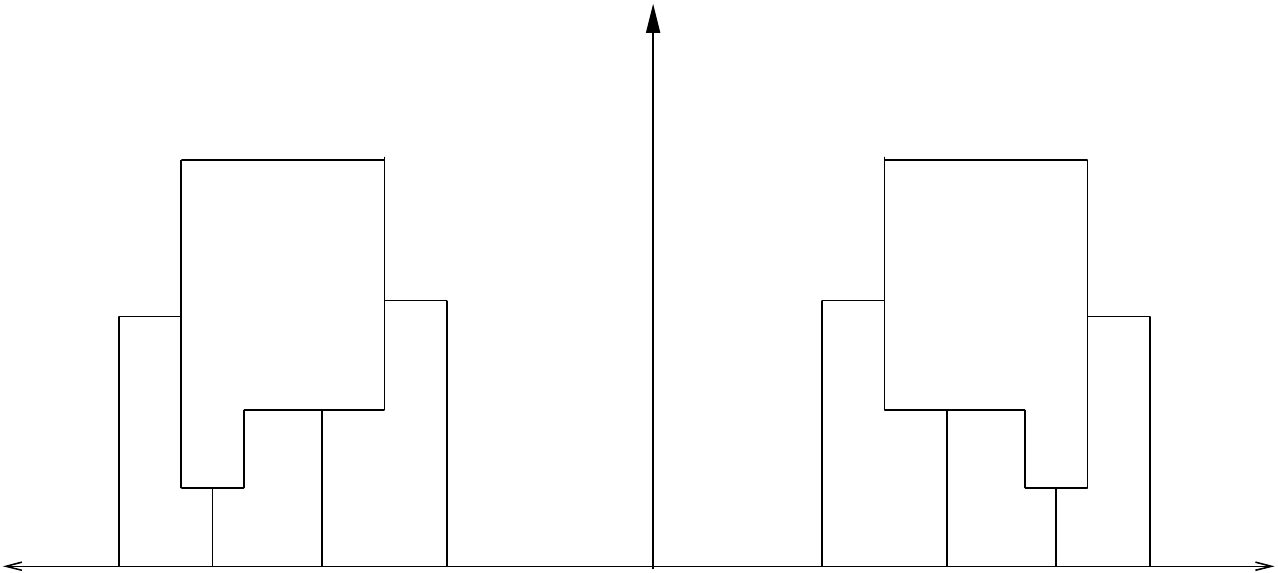}}
\put(205, 165){$i\mathbb{R}$}
\put(375,10){$\mathbb{R}$}
\put(75,68){$\mathcal{R}$}
\put(275,68){$-\overline{\mathcal{R}}$}
\put(220,68){$\Gamma$}
\put(260,16){$S$}
\end{picture}
\end{center}
\caption{}
\end{figure}

The precise construction of the function $V$ in this case, can be obtained by simply symmetrizing the construction appearing in \cite{TreilStable}.  For ease of the reader, we include the details.

The slits constructed divide $\partial\mathcal{R}$ (or, equivalently, the boundary $\p\mathcal{R}_{\delta'}$ of $\mathcal{R}_{\delta'}$) into arcs with hyperbolic length bounded by some constant, depending only on an absolute constant.  Note that we have a similar decomposition for the boundary of $-\overline{\mathcal{R}}$ due to the symmetry of the construction above.

Next, recall that for $H^\infty(\C_+)$ functions we have $\abs{f'(z)}\leq\f{\norm{f}_\infty}{\im z}$, and since $(\log f)'=\f{f'}{f}$, we will have 
$$
\abs{\log B_{\mathcal{R}}(z)B_{-\overline{\mathcal{R}}}(z)-\log B_{\mathcal{R}}(\xi)B_{-\overline{\mathcal{R}}}(\xi)}\leq C(\delta,\epsilon)
$$
for any $z,\xi\in\p\mathcal{R}_{\delta'}$.  The slits split the set $\C_+\setminus(\mathcal{R}_{\delta'}\cup-\overline{\mathcal{R}}_{\delta'})$ into connected components.  We then define, in each such domain $E$, the function $\vf$ as a branch of $\log B_{\mathcal{R}}B_{-\overline{\mathcal{R}}}$ for which $0=\textnormal{Im} \log B_\mathcal{R}(z)B_{-\overline{\mathcal{R}}}(z)$ if $z\in\textnormal{clos} E\cup\p\mathcal{R}_{\delta'}$.  Further note that on $i\R$ both above and below the domain, we will have the property that $B_\mathcal{R}B_{-\overline{\mathcal{R}}}$ is positive since there are an even number of zeros.  Note that by construction the function $\varphi$ will be real symmetric.  The jumps of $\vf$ on the slits, and on the boundary of $\mathcal{R}_{\delta'}\cup -\overline{\mathcal{R}}_{\delta'}$, are bounded by a constant $C(\delta,\epsilon)$.  Let $\Gamma^{\delta'}_{\mathcal{R}}$ denote the hyperbolic $\f{\delta'}{100}$--neighborhood of $\partial\mathcal{R}_{\delta'}\cup \partial\left(-\overline{\mathcal{R}}_{\delta'}\right)$, i.e., 
$$
\Gamma^{\delta'}_{\mathcal{R}}:=\left\{z\in\C_+:\inf_{a\in\partial\mathcal{R}_{\delta'}\cup \partial\left(-\overline{\mathcal{R}}_{\delta'}\right)}\abs{b_a(z)}<\f{\delta'}{100}\right\}.
$$
Also, for a slit $S$ of altitude $d$ let $S_{\delta'}$ be its $\f{\delta'}{100}$--neighborhood (usual \textit{not} hyperbolic) of $S$, i.e.,
$$
S_{\delta'}:=\left\{z\in\C_+:\textnormal{dist}(z,S)<\f{\delta'}{100}\right\}.
$$

Because of the trivial estimate $\abs{f'(z)}\leq\f{\norm{f}_{\infty}}{\im z}$ for $f\in H^\infty(\C_+)$, and since $\vf'=(\log B_{\mathcal{R}}B_{-\overline{\mathcal{R}}})'$, if $z\in\Gamma^{\delta'}_{\mathcal{R}}$ we have
$$
\abs{\vf'(z)}\leq\f{C(\delta,\epsilon)}{\im z}\quad\textnormal{for }z\in\Gamma^{\delta'}\setminus\left(\p\mathcal{R}_{\delta'}\cup\p \left(-\overline{\mathcal{R}}_{\delta'}\right)\right).
$$

Since the Blaschke product $B_\mathcal{R}B_{-\mathcal{R}}$ has no zeros is $\C_+\setminus(\mathcal{R}\cup-\overline{\mathcal{R}})$ has no zeros is $\C_+\setminus\left(\mathcal{R}\cup -\overline{\mathcal{R}}\right)$, it is analytic on the set $\C\setminus\overline{\mathcal{R}}$ (the bar denotes complex conjugation), and therefore, 
$$
\abs{(B_{\mathcal{R}}B_{-\overline{\mathcal{R}}})'(z)}\leq\f{C}{\textnormal{dist}(z,\mathcal{R}\cup-\overline{\mathcal{R}})}.
$$
Hence, for any slit $S$ we have
$$
\abs{\vf'(z)}\leq\f{C(\delta,\epsilon)}{d}\quad\textnormal{for } z\in S_{\delta'}\setminus S.
$$
These estimates and the boundedness of the jumps of $\vf$ allow one to change the function $\vf$ on the set $\Gamma^{\delta'}_{\mathcal{R}}\cup\bigcup_{S\in\mathcal{S}}S_{\delta'}$, where $\mathcal{S}$ denotes the collection of all slits for the component $\mathcal{R}$, to obtain a function $V_\mathcal{R}$ satisfying
\begin{displaymath}
V_\mathcal{R}=\vf(z)\quad\forall z\notin\Gamma^{\delta'}_{\mathcal{R}}\cup\left(\bigcup_{S\in\mathcal{S}}S_{\delta'}\right),
\end{displaymath}
\begin{displaymath}
\abs{V_\mathcal{R}'(z)}\leq\f{C(\delta,\epsilon)}{\im z},\qquad \abs{\Delta V_\mathcal{R}(z)}\leq\f{C(\delta,\epsilon)}{(\im z)^2}\quad\forall z\in\Gamma^{\delta'}_{\mathcal{R}},
\end{displaymath}
\begin{displaymath}
\abs{V_\mathcal{R}'(z)}\leq\f{C(\delta,\epsilon)}{d},\qquad \abs{\Delta V_\mathcal{R}(z)}\leq\f{C(\delta,\epsilon)}{d^2}\quad\forall z\in S_{\delta'}.
\end{displaymath}
The function $V_\mathcal{R}$ will be smooth by taking the convolution of $\vf$ with an appropriate smooth kernel.

\begin{rem}
Note that the construction in this case handles the situation when $\mathcal{R}$ is `far' from the imaginary axis.  If the region $\mathcal{R}$ is very close to the imaginary axis, then when we consider the regions $\mathcal{R}_{\delta'}$ and $-\overline{\mathcal{R}}_{\delta'}$ then it could be the case they over lap.  If that is the situation, then the argument used to construct the approximating function $V_{\mathcal{R}}$ will actually be given by the case when there are an even number of zeros in the domain $\mathcal{R}\cap i\R$ (since there are actually none).  In that case the argument given below will handle the construction.
\end{rem}

\subsection{Zeros on the Imaginary Axis}

We are left handling the case where there are zeros in $\sigma_1\cap i\R$ and when $\mathcal{R}=-\overline{\mathcal{R}}$.  In these cases we have to pair them appropriately to guarantee that the appropriate branch of the logarithm will exist.  This was overlooked in the original version \cite{W}.

Note that if we had $\mathcal{R}=-\overline{\mathcal{R}}$ and there were an \textit{even} number of zeros of the function $p$ in $i\R\cap\mathcal{R}$, then the argument given above works to handle this case.  Indeed, we have that $B_{\mathcal{R}}\in H^\infty_{\R}(\C_+)$ due to the symmetry of the domain $\mathcal{R}$.  Also, since the domain $\mathcal{R}$ has an even number of zeros in $i\R$, it is obvious that the function $B_{\mathcal{R}}$ is positive on $i\R$ both above and below the region $\mathcal{R}$.  Then one repeats the construction given above, or more precisely the one found in \cite{TreilStable}, and due to the symmetry of $\mathcal{R}$ and the even number of zeros in $i\R$, we obtain the function $V_{\mathcal{R}}$ satisfying the necessary properties.

We defined 
\begin{displaymath}
\vf(z):=\left\{
\begin{array}{lll} 
0 & : & z\in \mathcal{R}_{\delta'}\\
\log (B_\mathcal{R}) & : & \C_+\setminus\mathcal{R}_{\delta'},
\end{array}\right.
\end{displaymath}
for some branch of the logarithm that is real symmetric.  Repeating the argument from above, it is possible to then split the domain $\C_+\setminus\mathcal{R}$ into connected components via slits and $\Gamma$-slits, and enlarge the slits by hyperbolic neighborhoods so that they are disjoint.  We can then define in each such domain a branch of logarithm that is analytic, bounded and real symmetric since $B_{\mathcal{R}}$ has an even number of zeros.

These estimates and the boundedness of the jumps of $\vf$ allow one to change the function $\vf$ on the set $\Gamma^{\delta'}_{\mathcal{R}}\cup\bigcup_{S\in\mathcal{S}}S_{\delta'}$, where $\mathcal{S}$ denotes the collection of all slits for the component $\mathcal{R}$, to obtain a function $V_\mathcal{R}$ satisfying
\begin{displaymath}
V_\mathcal{R}=\vf(z)\quad\forall z\notin\Gamma^{\delta'}_{\mathcal{R}}\cup\left(\bigcup_{S\in\mathcal{S}}S_{\delta'}\right),
\end{displaymath}
\begin{displaymath}
\abs{V_\mathcal{R}'(z)}\leq\f{C(\delta,\epsilon)}{\im z},\qquad \abs{\Delta V_\mathcal{R}(z)}\leq\f{C(\delta,\epsilon)}{(\im z)^2}\quad\forall z\in\Gamma^{\delta'}_{\mathcal{R}},
\end{displaymath}
\begin{displaymath}
\abs{V_\mathcal{R}'(z)}\leq\f{C(\delta,\epsilon)}{d},\qquad \abs{\Delta V_\mathcal{R}(z)}\leq\f{C(\delta,\epsilon)}{d^2}\quad\forall z\in S_{\delta'}.
\end{displaymath}
The function $V_\mathcal{R}$ will be smooth by taking the convolution of $\vf$ with an appropriate smooth kernel.

However, in the case where the region $\mathcal{R}$ has an odd number of zeros in $i\R$ this construction fails to produce a real symmetric branch of the logarithm.  The reason for this is the following the corresponding Blaschke product $B_\mathcal{R}$ will have a different sign on $i\R$ above and below $\mathcal{R}$.   Where it is negative, the imaginary part of logarithm should be $\pi +2\pi k$, while any continuous symmetric branch of logarithm should have imaginary part being zero.  To remedy this, we must appropriately pair the domains $\mathcal{R}$ and the zeros from $\sigma_1\cap i\R$ to produce regions with an even number of zeros in $i\R$.  Once we have constructed these new regions, then we can (essentially) apply the observation about constructing the approximating function when there were an even number of zeros.

Recall that \eqref{finiteUnion} gives us
$$
\{iy:|q(iy)|<\delta'\}=\bigcup_{l=1}^N Z_l
$$
where $Z_l$ is are open disjoint intervals contained in $i\R$.  Now, a simple observation, given a symmetric connected region $\mathcal{R}$ then the region must intersect the imaginary axis $i\R$ in an interval, and the interval must lie between the two consecutive intervals $Z_{l_1}$ and $Z_{l_2}$.  Indeed, must have that $p(iy)\geq C(\delta, \epsilon)$ on each of the intervals $Z_{l_1}$ and $Z_{l_2}$, but by the construction of the region $\mathcal{R}$, we had to have that for $iy\in\mathcal{R}$ that $|p(iy)|<\delta'$.  Thus, if the region $\mathcal{R}$ contains an odd number of zeros in $i\R$, there must be another ``region'' $\mathcal{R'}$ that contains an odd number of zeros that also intersects $i\R$ between the intervals $Z_{l_1}$ and $Z_{l_2}$ (here a ``region'' is either an $\mathcal{R}$ or a point in $\sigma_1$).  Note that this implies that between consecutive intervals $Z_{l_1}$ and $Z_{l_2}$ there must be an \textit{even} number of these regions containing an \textit{odd} number of zeros in $i\R$.  This implies that we can pair these consecutive regions with an odd number of zeros, producing regions with even numbers of zeros in $i\R$.  In doing this, we will be able to construct the the logarithm that is real symmetric.

We now describe how to pair these regions so that we can define the appropriate logarithm.  

\subsubsection{Connecting regions \texorpdfstring{$\mathcal{R}$}{R} and \texorpdfstring{$\mathcal{R}'$}{R'} with an ``odd'' number of zeros}
First, suppose that we have to pair two regions of the form $\mathcal{R}$ and $\mathcal{R}'$ that are symmetric connected components between the intervals $Z_{l_1}$ and $Z_{l_2}$.  Now by the construction each of these domains is symmetric, and the projection of one of these domains onto $\R$ is wholely contained in the other.  So, without loss of generality, we can assume that $\mathcal{R}$ lies above $\mathcal{R}'$. 

Now note that the function $B_{\mathcal{R}}B_{\mathcal{R}'}\in H^\infty_\R(\C_+)$ since the domains $\mathcal{R}$ and $\mathcal{R}'$ are symmetric.  We also have that the function $B_{\mathcal{R}}B_{\mathcal{R}'}$ is positive on $i\R$ above the domain $\mathcal{R}$ and below the domain $\mathcal{R}'$.  We then connect these two domains via a slit $I$ on the imaginary axis $i\R$.  We can then take all the corresponding slits from the region $\mathcal{R}$, except the slit on the imaginary axis connecting $\mathcal{R}$ and $\R$.  We can also then take the slits for the domain $\mathcal{R}'$ and this decomposes the domain $\C_+\setminus\left(\mathcal{R}\cup\mathcal{R}'\cup I\right)$ into connected components.  By construction of the slits and the domains $\mathcal{R}$ and $\mathcal{R}'$ we have that the slits are disjoint, and similarly the $\delta'$ neighborhoods of the slits will be disjoint.  We are now in the setup where we have a symmetric region with an even number of zeros and so the discussion above holds.

We then define, in a similar manner as above, 
\begin{displaymath}
\vf(z):=\left\{
\begin{array}{lll} 
0 & : & z\in \mathcal{R}_{\delta'}\cup\mathcal{R}'_{\delta'}\cup I_{\delta'}\\
\log (B_\mathcal{R}B_{\mathcal{R}'}) & : & \C_+\setminus\left(\mathcal{R}_{\delta'}\cup\mathcal{R}'_{\delta'}\cup I_{\delta'}\right),
\end{array}\right.
\end{displaymath}
for an appropriate branch of the logarithm that is real symmetric.  Repeating the argument from above, it is possible to then split the domain $\C_+\setminus(\mathcal{R}_{\delta'}\cup\mathcal{R}'_{\delta'}\cup I_{\delta'})$ into connected components via slits and $\Gamma$-slits, and enlarge the slits by hyperbolic neighborhoods so that they are disjoint.  We can then define in each such domain a branch of logarithm that is analytic, bounded and real symmetric since $B_{\mathcal{R}}B_{\mathcal{R}'}$ has an even number of zeros on $i\R$.  The construction is explained in the figure below.\\

\begin{figure}[htbp]
\label{region1_2}
\begin{center}
\begin{picture}(400,200)
\put(0,0){\includegraphics{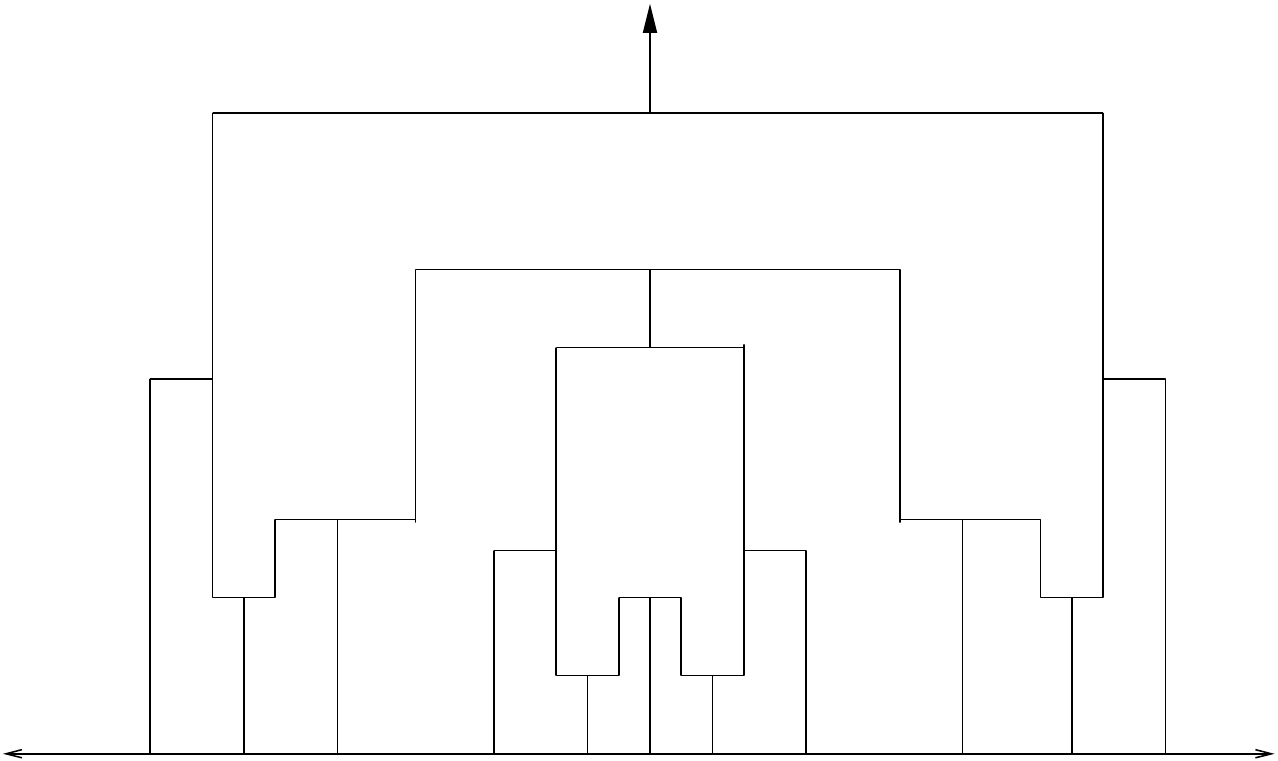}}
\put(205, 205){$i\mathbb{R}$}
\put(375,10){$\mathbb{R}$}
\put(185,168){$\mathcal{R}$}
\put(185,68){$\mathcal{R}'$}
\put(195,128){$I$}
\put(220,68){$\Gamma$}
\put(265,16){$S$}
\end{picture}
\end{center}
\caption{}
\end{figure}

As before, let $\mathcal{S}$ denotes the collection of all slits for the component $\mathcal{R}$ and $\mathcal{R}'$ and the slit joining these regions together (except the slit connecting $\mathcal{R}$ to $\R$) and let $\Gamma^{\delta'}_{\mathcal{R},\mathcal{R}'}$ denote the $\frac{\delta'}{100}$ hyperbolic neighborhood of the region $\mathcal{R}\cup\mathcal{R}'\cup I$.  These estimates and the boundedness of the jumps of $\vf$ allow one to change the function $\vf$ on the set $\Gamma^{\delta'}_{\mathcal{R},\mathcal{R}'}\cup\bigcup_{S\in\mathcal{S}}S_{\delta'}$ to obtain a function $V_{\mathcal{R},\mathcal{R}'}$ satisfying
\begin{displaymath}
V_{\mathcal{R},\mathcal{R}'}=\vf(z)\quad\forall z\notin\Gamma^{\delta'}_{\mathcal{R},\mathcal{R}'}\cup\left(\bigcup_{S\in\mathcal{S}}S_{\delta'}\right),
\end{displaymath}
\begin{displaymath}
\abs{V_{\mathcal{R},\mathcal{R}'}'(z)}\leq\f{C(\delta,\epsilon)}{\im z},\qquad \abs{\Delta V_{\mathcal{R},\mathcal{R}'}(z)}\leq\f{C(\delta,\epsilon)}{(\im z)^2}\quad\forall z\in\Gamma^{\delta'}_{\mathcal{R},\mathcal{R}'},
\end{displaymath}
\begin{displaymath}
\abs{V_{\mathcal{R},\mathcal{R}'}'(z)}\leq\f{C(\delta,\epsilon)}{d},\qquad \abs{\Delta V_{\mathcal{R},\mathcal{R}'}(z)}\leq\f{C(\delta,\epsilon)}{d^2}\quad\forall z\in S_{\delta'}.
\end{displaymath}
As before, to obtain the desired function $V_{\mathcal{R},\mathcal{R}'}$, simply convolve the function $\vf$ with an appropriate symmetric smooth kernel.

\subsubsection{Connecting a region $\mathcal{R}$ with a zero in $\sigma_1$}

The next case is when we have to pair a region $\mathcal{R}$ with and odd number of zeros in $i\R$ and a zero $a\in\sigma_1\cap i\R$.  We describe the case when the zero in $\sigma_1\cap i\R$ lies below the domain $\mathcal{R}$ since the opposite situation is similar.  
We connect the region $\mathcal{R}_{\delta'}$ and the region $D_{a}$ by a slit $I_{\mathcal{R},a}$ on the imaginary axis.  We then take the slits corresponding to the region $\mathcal{R}$, except for the slit on the imaginary axis connecting the region $\mathcal{R}$ with $\R$.  We finally connect the boundary of $D_a$ to the real axis by a slit $I$ on the imaginary axis.

Now consider the function $B_{\mathcal{R}}b_{a}\in H^\infty_\R(\C_+)$ which has an even number of zeros, and so, when we consider the simply connected regions given by the splitting of $\C_+\setminus(\mathcal{R}_{\delta'}\cup D_a\cup I_{\delta'})$  using the appropriate slits, we can define a real symmetric branch of the logarithm.    

In this situation, we now define 
\begin{displaymath}
\vf(z):=\left\{
\begin{array}{lll} 
0 & : & z\in \mathcal{R}_{\delta'}\cup D_a\cup I_{\delta'}\\
\log (B_\mathcal{R}b_a) & : & \C_+\setminus\left(\mathcal{R}_{\delta'}\cup D_a \cup I_{\delta'}\right),
\end{array}\right.
\end{displaymath}
where we choose a branch of the logarithm that is real symmetric.  A repetition of the argument from above, shows that it is possible to then split the domain $\C_+\setminus(\mathcal{R}_{\delta'}\cup D_a\cup I_{\delta'})$ into connected components via slits and $\Gamma$-slits, and enlarge the slits by hyperbolic neighborhoods so that they remain disjoint.  We can then define in each such domain a branch of logarithm that is analytic, bounded and real symmetric since $B_{\mathcal{R}}b_a$ has an even number of zeros on $i\R$.

As above we let $\mathcal{S}$ denotes the collection of all slits for the component $\mathcal{R}$ (with out the slit joining $\mathcal{R}$ to $\R$) and the slit joining $\mathcal{R}$ and $D_a$ and the slit joining together $D_a$ and $\R$.  Let $\Gamma^{\delta'}_{\mathcal{R}}$ denote the $\frac{\delta'}{100}$ hyperbolic neighborhood of the region $\mathcal{R}$.  Let $T_a^{\delta'}$ denote the $\delta'\im a$ neighborhood of $T_a$.  The estimates from above and the boundedness of the jumps of $\vf$ allow one to change the function $\vf$ on the set $\Gamma^{\delta'}_{\mathcal{R},\mathcal{R}'}\cup\bigcup_{S\in\mathcal{S}}S_{\delta'}\cup T_a^{\delta'}$ to obtain a function $V_{\mathcal{R},a}$ satisfying
\begin{displaymath}
V_{\mathcal{R},a}=\vf(z)\quad\forall z\notin\Gamma^{\delta'}_{\mathcal{R}}\cup\left(\bigcup_{S\in\mathcal{S}}S_{\delta'}\right)\cup T_a^{\delta'},
\end{displaymath}
\begin{displaymath}
\abs{V_{\mathcal{R},a}'(z)}\leq\f{C(\delta,\epsilon)}{\im z},\qquad \abs{\Delta V_{\mathcal{R},a}(z)}\leq\f{C(\delta,\epsilon)}{(\im z)^2}\quad\forall z\in\Gamma^{\delta'}_{\mathcal{R}},
\end{displaymath}
\begin{displaymath}
\abs{V_{\mathcal{R},a}'(z)}\leq\f{C(\delta,\epsilon)}{d},\qquad \abs{\Delta V_{\mathcal{R},a}(z)}\leq\f{C(\delta,\epsilon)}{d^2}\quad\forall z\in S_{\delta'}.
\end{displaymath}
Convolution of $\vf$ produces the smooth function $V_{\mathcal{R},a}$ that we seek.

\subsubsection*{Connecting two regions corresponding to zeros in $\sigma_1$}

The final case is where we have to connect two real zeros in $\sigma_1$.  Let $a, a'$ denote the zeros labeled so that $a$ is the point closest to the real axis.  About $a$ let $D_{a}$ denote the disc with center at the point $a$ of radius $\delta'\im a$, and let $T_{a}=\p D_{a}$.  Repeat this construction for the point $a'$.  Then connect these discs together by drawing a vertical slit on the imaginary axis denoted $I_{a, a'}$ along the imaginary axis connecting $D_{a}$ to $D_{a'}$.  Finally, draw the slit $I_{0,a}$ connecting the boundary of the disc $D_{a}$ with the real axis.  The construction is explained in Figure \hyperref[tworealzeros]{5} below.\\

\begin{figure}[htbp]
\label{tworealzeros}
\begin{center}
\begin{picture}(400,200)
\put(0,0){\includegraphics{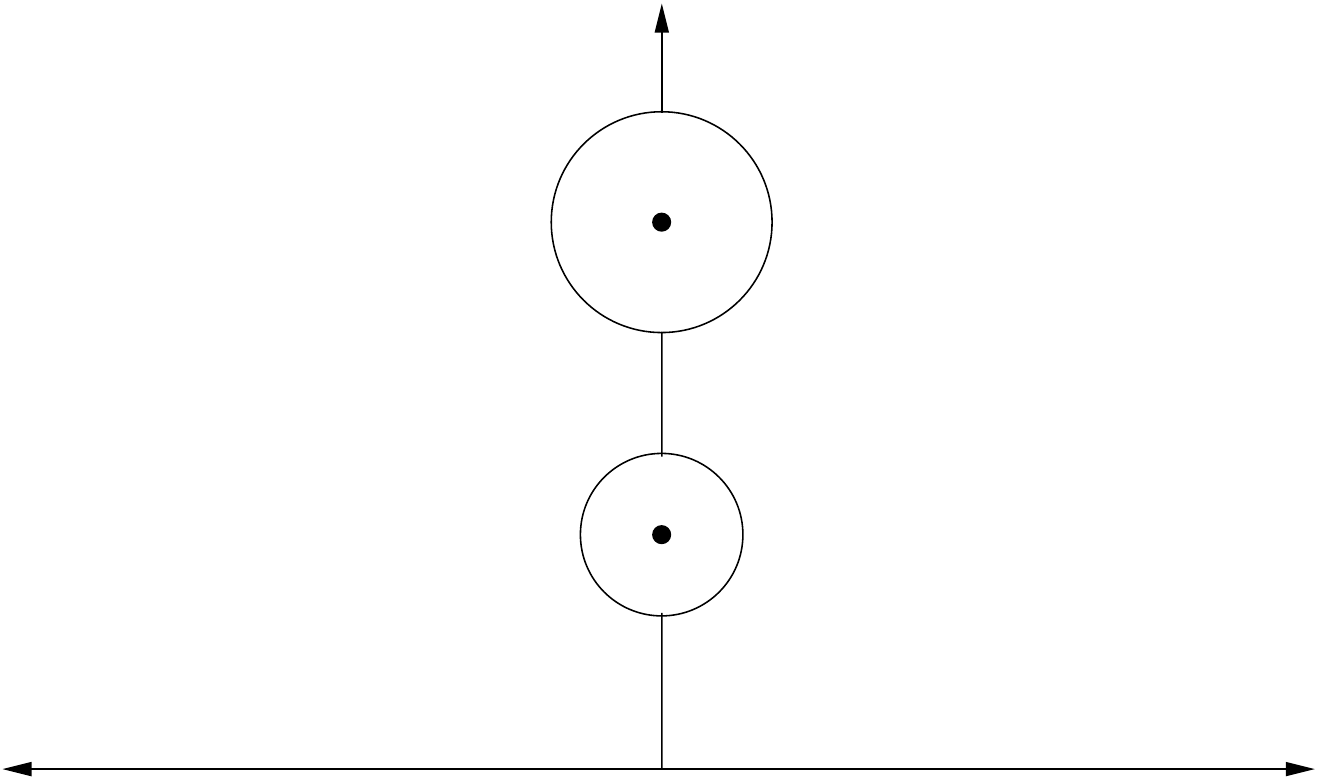}}
\put(205, 225){$i\mathbb{R}$}
\put(375,10){$\mathbb{R}$}
\put(200,68){$a$}
\put(200,160){$a'$}
\put(200,108){$I_{a,a'}$}
\put(230,68){$T_{a}$}
\put(230,160){$T_{a'}$}
\end{picture}
\end{center}
\caption{}
\end{figure}

Then, one notes that $\abs{b_{a}b_{a'}}\leq C(\delta,\epsilon)$ on $D_{a}\cup D_{a'}\cup I_{a,a'}$ and $\abs{b_{a}b_{a'}}\geq C(\delta,\epsilon)$ for points in the compliment.  We then define $\vf$ in the following manner,
\begin{displaymath}
\vf(z):=\left\{
\begin{array}{lll} 
0 & : & z\in D_{a}\cup D_{a'}\cup I_{a,a'}\\
\log \left(b_{a}b_{a'}\right) & : & \textnormal{otherwise}.
\end{array}\right. 
\end{displaymath}
Note that we have $b_ab_{a'}$ is real and positive on $i\R$ above the point $a'+\delta'\im a'$ and below $a-\delta'\im a$.  Clearly we have that $\varphi$ is analytic, bounded (with a bound depending on $\delta$ and $\epsilon$) and real symmetric.  We again smooth $\vf$ to find our function $V$.  Namely, we change $\vf$ in a $\f{\delta'}{100}\min\{\im a, \im a'\}$ neighborhood of $I_{0,a}\cup T_{a}\cup I_{a,a'}\cup T_{a'}$ to obtain a smooth function on $\C_+$ such that:
\begin{itemize}
\item[(i)] $\abs{\pbar V_{a,a'}(z)}\leq C(\delta,\epsilon)\min\left\{\frac{1}{\im a'},\frac{1}{\im a}\right\}$;\\
\item[(ii)] $\abs{\Delta V_{a,a'}(z)}\leq C(\delta,\epsilon)\min\left\{\frac{1}{(\im a')^2},\frac{1}{(\im a)^2}\right\}$;\\
\item[(iii)] $V_{a_1,a_2}(z)=\vf(z)$ if $\textnormal{dist}\left(z, T_{a}\cup I_{a,a'}\cup T_{a'}\right)>\f{\delta'}{100}\min\{\im a, \im a'\}$;\\
\item[(iv)] $V_{a,a'}(z)=\overline{V_{a,a'}(-\overline{z})}$.
\end{itemize}

\begin{rem}
We remark that the constructions in the cases above, in certain situations the construction could be greatly simplifed.  For example, in the case of of $a,a'\in\sigma_1\cap i\R$ it could be the case that the discs $D_a$ and $D_{a'}$ could overlap, and so the resulting slit connecting those discs on the imaginary axis $i\R$ could be ignored.  The construction would then proceed in a similar manner as the original construction of $V_{a}$ with $a\in\sigma_1\cap(\C_+\setminus i\R)$.  A similar remark applies in the case of $\mathcal{R}=-\overline{\mathcal{R}}$, $\mathcal{R}'=-\overline{\mathcal{R}'}$ with $\mathcal{R}\cap\mathcal{R}'\cap i\R$ having an even number of zeros, and when we have $\mathcal{R}=-\overline{\mathcal{R}}$ with $\mathcal{R}\cap i\R$ having an odd number of zeros and $a\in\sigma_1\cap i\R$.  

The constructions in these cases are easier since they are simply repeats of the ideas and arguments given above.  For brevity, we have omitted these construction in these cases and focused on the more ``general'' situation.
\end{rem}

\subsection{Definition of the Function $V$}

The function $V$ is then defined as the sum of all summands of the the zeros on the imaginary axis, and the zeros off the imaginary axis
$$
V=\sum_{a\in\sigma_1\cap\left(\C_+\setminus i\R\right)} V_a+\sum_{\stackrel{\mathcal{R}\in\mathfrak{R}}{\mathcal{R}\neq -\overline{\mathcal{R}}}} V_\mathcal{R}+\sum_{a,a'\in\sigma_1\cap i\R} V_{a,a'}+\sum_{\stackrel{\mathcal{R}\in\mathfrak{R}}{\mathcal{R}=-\overline{\mathcal{R}}}} V_{\mathcal{R},\mathcal{R}'}+\sum_{\stackrel{\mathcal{R}\in\mathfrak{R},\mathcal{R}=-\overline{\mathcal{R}}}{a\in\sigma_1\cap i\R}} V_{\mathcal{R},a}.
$$
For simplicity in this display we have suppressed the distinction between $\mathcal{R}$ having an even or odd number of zeros in $\mathcal{R}\cap i\R$.  In the construction, there were two estimates that arose, those from terms like $V_a$ and those from terms like $V_{\mathcal{R}}$.  When we prove the remaining properties below, we can just appeal to the appropriate estimates.

By construction, $V$ will be real symmetric.  Therefore, it only remains to show that $V$ satisfies the required conditions (i), (ii) and (iv) from Section \ref{ideaofproof}.

\section{\texorpdfstring{Verification of the Properties of the Function $V$}{Verification of the Properties of the Function V}}
\label{Visgood}

With $V$ now constructed, we need only show that it possesses all the required properties.  First, we need an auxiliary definition and some propositions.  These lemmas and propositions are taken from Treil's paper \cite{TreilStable}.  We omit the proofs.

Let $S$ be a slit corresponding to a component $\mathcal{R}$ (if $\mathcal{R}\neq-\overline{\mathcal{R}}$ or $\mathcal{R}\cup\mathcal{R}'$ if $\mathcal{R}=-\overline{\mathcal{R}}$, $\mathcal{R}=-\overline{\mathcal{R}'}$ and $\mathcal{R}\cap\mathcal{R'}\cap i\R$ has an even number of zeros.  A common point of $S$ and $\p\mathcal{R}$ will be called an \textit{origin} of $S$.  For a slit corresponding to a point $a\in\sigma_1$, we shall call the origin of $S$ simply the point $a$.  In the case when $a,a'\in\sigma_1\cap i\R$, the origin of the slit corresponding to $a$ will be the point $a$, while the origin of slit that joined the two circles $D_a$ and $D_{a'}$ together will be $a'$.

The following lemma is a straightforward application of Proposition \ref{CarlesonMeasures}.

\begin{lm}
\label{lm1}
Let $A$ denote the set of origins of all the slits constructed above.  Then the measure $\sum_{a\in A}\im a\, \delta_a$ is a Carleson measure with intensity at most $C(\delta,\epsilon)$.
\end{lm}

Let $S$ be a slit constructed above, and let $d=d(S)$ be its altitude.  An integer $k$ will be called the \textit{rank} of $S$ and denoted $\rk(S)$ if $2^k\leq d< 2^{k+1}$.  Of course the rank of a slit can be negative.

\begin{lm}
\label{lm2}
For a given $z\in\C_+$ and $k\in\Z$ the number of slits of rank $k$ for which $z\in S_{\delta'}$ is at most $C(\delta,\epsilon)$.
\end{lm}

\begin{lm}
\label{lm3}
For a given $z\in\C_+$, there exists at most $C(\delta,\epsilon)$ components $\mathcal{R}\in\mathfrak{R}$ such that the hyperbolic $\f{\delta'}{100}$--neighborhood $\Gamma^{\delta'}_\mathcal{R}$ of $\p\mathcal{R}_{\delta'}$ contains the point $z$.
\end{lm}

For $a\in\sigma_1$, we constructed circles $T_a$, or $T_{a,-\overline{a}}$ about the point $a,-\overline{a}$ with $T_a$ having a radius $\delta'\im a$.  We let $T_a^{\delta'}$ and  $T_{a,-\overline{a}}^{\delta'}$ be the $\delta'\im a$ neighborhood of the appropriate region.  We have the following proposition.

\begin{lm}
\label{lm4}
For a given $z\in\C_+$ there exist at most $C(\delta,\epsilon)$ points $a\in\sigma_1$ such that $z\in T_a^{\delta'}$, $T_{a,-\overline{a}}^{\delta'}$.
\end{lm}

We remark that these lemmas follow from Treil's original proof, and then the observation that we have joined certain regions and so might be able to do slightly better than the resulting constant $C(\delta,\epsilon)$.

To show that the function $V$ satisfies the inequality (i) of Proposition \ref{CorrectingV_new}, we will use a result due to Treil, \cite{TreilHankel}.  

\begin{lm}
\label{TreilLemma}
Let $0<\epsilon<1$.  Let $\Theta_n$ $(n\in\N)$ be inner functions and suppose $\gamma_n$ is a ``semi-Carleson contour'' for $\Theta_n$, i.e., $\gamma_n=\p\mathcal{V}_n$, where $\mathcal{V}_n$ is an open set $\mathcal{V}_n\supset\{z:\abs{\Theta_n(z)}<\epsilon\}$.  Moreover, the measure $l_{\gamma_n}$ (arclength on $\gamma_n$) is a Carleson measure with intensity at most $C$, with $C$ independent of $n$.  Suppose that the measure $\sum_nl\gamma_n$ is a Carleson measure with intensity at most $C_1$.  Then
$$
\sum_{n\in\N}\left(1-\abs{\Theta_n(z)}^2\right)\leq KC_1,\ \forall z\in \C_+,\ \textnormal{where } K=K(\epsilon,C).
$$
\end{lm}

This Lemma is applied to the family of functions made from the union of the following two families.  The first family corresponds to those having no zeros on the imaginary axis.
$$
\left\{B_\mathcal{R}B_{-\overline{\mathcal{R}}}:\mathcal{R}\in\mathfrak{R}, \ \mathcal{R}\neq-\overline{\mathcal{R}}\right\}\cup \left\{b_ab_{-\overline{a}}:a\in\sigma_1\cap\left(\C_+\setminus i\R\right)\right\}
$$
While the second family corresponds to the case of zeros on the imaginary axis.  There are four sub-families in this case:
$$
\left\{B_\mathcal{R}:\mathcal{R}\in\mathfrak{R}, \mathcal{R}=-\overline{\mathcal{R}}, \textnormal{ even number of zeros in } \mathcal{R}\cap i\R\right\},
$$
$$
\left\{B_\mathcal{R}B_{\mathcal{R}'}:\mathcal{R},\mathcal{R}'\in\mathfrak{R}, \mathcal{R}=-\overline{\mathcal{R}},\mathcal{R}'=-\overline{\mathcal{R}'}, \textnormal{ odd number of zeros in } \mathcal{R}\cap i\R \textnormal{ and } \mathcal{R}'\cap i\R\right\},
$$
$$
\left\{b_a b_{a'}:a,a'\in\sigma_1\cap i\R\right\},
$$ 
and
$$
\left\{ B_{\mathcal{R}}b_a: a\in\sigma_1\cap i\R, \mathcal{R}\in\mathfrak{R}, \mathcal{R}=-\overline{\mathcal{R}}, \textnormal{ odd number of zeros in } \mathcal{R}\cap i\R\right\}.
$$
For the case of the $B_\mathcal{R}B_{-\overline{\mathcal{R}}}$, we take for the semi-Carleson contour the boundary $\p\mathcal{R}_{\delta'}\cup\p\left(-\overline{\mathcal{R}}_{\delta'}\right)$.  In the case $B_\mathcal{R}$ with $\mathcal{R}$ symmetric and having an even number of zeros in $\mathcal{R}\cap i\R$ we simply take $\p\mathcal{R}_{\delta'}$.  In the case of $B_{\mathcal{R}}B_{\mathcal{R}'}$ with $\mathcal{R}$ and $\mathcal{R}'$ symmetric and having an odd number of zeros in $\mathcal{R}\cap i\R$ and $\mathcal{R}\cap i\R$ we take $\p\mathcal{R}_{\delta'}\cup \p\mathcal{R}'_{\delta'}$.  If we have $\mathcal{R}$ symmetric with an odd number of zeros in $\mathcal{R}\cap i\R$ and $a\in\sigma_1\cap i\R$ we take $\partial\mathcal{R}_{\delta'}\cup T_{a}$.  For the other cases, we take either $T_a\cup T_{-\overline{a}}$, $T_{a,-\overline{a}}$ or $T_{a}\cup T_{a'}$.  The assumption of Lemma \ref{TreilLemma} follows from Proposition \ref{CarlesonMeasures}.  Therefore, we have
\begin{eqnarray*}
\sum_{\substack{\mathcal{R}\in\mathfrak{R}\\ \mathcal{R}\neq-\overline{\mathcal{R}}}}(1-\abs{B_\mathcal{R}B_{-\overline{\mathcal{R}}}}^2)\ +\ \sum_{\substack{\mathcal{R}\in\mathfrak{R}\\ \mathcal{R}=-\overline{\mathcal{R}}}}(1-\abs{B_\mathcal{R}}^2) \ + \sum_{\substack{\mathcal{R},\mathcal{R}'\in\mathfrak{R}\\ \mathcal{R}=-\overline{\mathcal{R}}, \mathcal{R}'=-\overline{\mathcal{R}'}}}(1-\abs{B_\mathcal{R}B_{\mathcal{R}'}}^2) &  & \\
\sum_{a\in\sigma_1\cap\left(\C_+\setminus i\R\right)}(1-\abs{b_a(z)b_{-\overline{a}}(z)}^2)\ + \ \sum_{a,a'\in\sigma_1\cap i\R}(1-\abs{b_a(z)b_{a'}(z)}^2) +\sum_{\stackrel{\mathcal{R}\in\mathfrak{R},\mathcal{R}=-\overline{\mathcal{R}}}{a\in\sigma_1\cap i\R}} (1-|B_{\mathcal{R}}b_a|^2)& \leq &  C(\delta,\epsilon).
\end{eqnarray*}

But, by the construction of $V$ in Section \ref{mainconstruction}, we have that (with the last three being estimates from the case of zeros on $i\R$)
$$
\begin{array}{clll}
\abs{\re V_\mathcal{R}(z)} & \leq & \min\left\{\log\f{1}{\delta'},\log\abs{B_{\mathcal{R}}B_{-\overline{\mathcal{R}}}}^{-1}\right\}, & \mathcal{R}\neq -\overline{\mathcal{R}}\\
\abs{\re V_a(z)} & \leq & \min\left\{\log\f{1}{\delta'},\log\abs{b_a(z)b_{-\overline{a}}(z)}^{-1}\right\}, & a\in\sigma_1\cap(\C_+\setminus i\R)\\
\abs{\re V_\mathcal{R}(z)} & \leq & \min\left\{\log\f{1}{\delta'},\log\abs{B_{\mathcal{R}}}^{-1}\right\}, & \textnormal{even number of zeros in } i\R\cap\mathcal{R}\\
\abs{\re V_{\mathcal{R},\mathcal{R}'}(z)} & \leq & \min\left\{\log\f{1}{\delta'},\log\abs{B_{\mathcal{R}}B_{\mathcal{R}'}}^{-1}\right\}, & \textnormal{even number of zeros in } i\R\cap\mathcal{R}\cap\mathcal{R}'\\
\abs{\re V_{a,a'}(z)} & \leq & \min\left\{\log\f{1}{\delta'},\log\abs{b_a(z)b_{a'}(z)}^{-1}\right\}, & a,a'\in \sigma_1\cap i\R\\
\abs{\re V_{\mathcal{R},a}(z)} & \leq & \min\left\{\log\f{1}{\delta'},\log\abs{B_{\mathcal{R}}(z)b_{a}(z)}^{-1}\right\}, & \mathcal{R}=-\overline{\mathcal{R}}, a\in\sigma_1\cap i\R
\end{array}
$$
for some constant $0<\delta'$.  Hence we have,
$$
\abs{\re V(z)}\leq C(\delta,\epsilon).
$$
We now prove that $V$ satisfies the conditions necessary to guarantee the existence of solutions to the $\db$-equation.  Namely, we prove that the Laplacian and the derivative of $V$ gives rise to a Carleson measure.

The proof again is basically a repeat of what appears in \cite{TreilStable}.  We include it for the ease of the reader.  Note that 
$$
\Delta V=\sum_{a\in\sigma_1\cap\left(\C_+\setminus i\R\right)} \Delta V_a+\sum_{\stackrel{\mathcal{R}\in\mathfrak{R}}{\mathcal{R}\neq -\overline{\mathcal{R}}}} \Delta V_\mathcal{R}+\sum_{a,a'\in\sigma_1\cap i\R} \Delta V_{a,a'}+\sum_{\stackrel{\mathcal{R}\in\mathfrak{R}}{\mathcal{R}=-\overline{\mathcal{R}}}} \Delta V_{\mathcal{R},\mathcal{R}'}+\sum_{\stackrel{\mathcal{R}\in\mathfrak{R},\mathcal{R}=-\overline{\mathcal{R}}}{a\in\sigma_1\cap i\R}} \Delta V_{\mathcal{R},a}.
$$
We have that the summand $\Delta V_\mathcal{R}(z)$ is not equal to zero (respectively, $\Delta V_a(z)\neq 0$), only if either
\begin{itemize}
\item[(a)] $z\in\Gamma^{\delta'}_\mathcal{R}$ (respectively, $z\in T_a^{\delta'}$ or $z\in T_{a, -\overline{a}}^{\delta'}$), or\\
\item[(b)] $z\in S_{\delta'}$ where $S$ is a slit corresponding to the component $\mathcal{R}$ (respectively, to the point $a$).
\end{itemize}
Similarly, the summand $\Delta V_{\mathcal{R},\mathcal{R}'}(z)$ is not equal to zero (respectively, $\Delta V_{a,a'}(z)\neq 0$), only if either
\begin{itemize}
\item[(c)] $z\in\Gamma^{\delta'}_{\mathcal{R},\mathcal{R}'}$ (respectively, $z\in T_{a,a'}^{\delta'}$), or\\
\item[(d)] $z\in S_{\delta'}$ where $S$ is a slit corresponding to the component $\mathcal{R}\cup\mathcal{R}'$ (respectively, to the points $a$ and $a'$).
\end{itemize}
Finally, the summand $\Delta V_{\mathcal{R},a}(z)$ is not equal to zero, only if
\begin{itemize}
\item[(e)] $z\in\Gamma^{\delta'}_{\mathcal{R}}\cup T_{a}^{\delta'}$, or\\
\item[(f)] $z\in S_{\delta'}$ where $S$ is a slit corresponding to the component $\mathcal{R}\cup D_a$.
\end{itemize}

By Lemmas \ref{lm3} and \ref{lm4}, for each $z\in\C_+$, at most $C(\delta,\epsilon)$ summands satisfy condition (a), (c) and (e).  Therefore, given the estimates obtained in the construction of $V_\mathcal{R}$, $V_a$, $V_{\mathcal{R},\mathcal{R}'}$, $V_{a,a'}$ and $V_{\mathcal{R},a}$ we have
\begin{eqnarray*}
\sum_{\mathcal{R}\in\mathfrak{R}, a\in\sigma_1: z\in \Gamma^{\delta'}_{\mathcal{R}}\cup T_{a}^{\delta'}} \abs{\Delta V_{\mathcal{R},a}(z)}+\sum_{\mathcal{R},\mathcal{R'}: z\in\Gamma^{\delta'}_{\mathcal{R},\mathcal{R}'}}\abs{\Delta V_{\mathcal{R},\mathcal{R}'}(z)}+\sum_{a,a':z\in T_{a,a'}^{\delta'}}\abs{\Delta V_{a,a'}(z)}\\
 +\sum_{\mathcal{R}: z\in\Gamma^{\delta'}_\mathcal{R}}\abs{\Delta V_\mathcal{R}(z)}+\sum_{a:z\in T_a^{\delta'}}\abs{\Delta V_a(z)}\leq\f{C(\delta,\epsilon)}{(\im z)^2}.
\end{eqnarray*}
It remains to deal with those points that contribute arising from condition (b), (d) and (f).  

Let $\mathrm{R}_k(z)$ be the collection of all components $\mathcal{R}$ such that there exists a slit $S$ of rank $k$ corresponding to the component for which $z\in S_{\delta'}$.  Similarly, define $\mathrm{A}_k(z)$ to be the set of all zeros $a\in\sigma_1$ for which $z\in S_{\delta'}$, $\rk S=k$, and $S$ is the slit corresponding to the point $a$.  Lemma \ref{lm2} implies that $\card\mathrm{R}_k(z)+\card\mathrm{A}_k(z)\leq C(\delta,\epsilon)$ for all $z\in\C_+$.  Note that the sets $\mathrm{R}_k(z)$ and $\mathrm{A}_k(z)$ are non-empty when $2^k>\f{1}{2}\im z$.



Now the estimates from the construction $V_\mathcal{R}$, $V_a$, $V_{\mathcal{R}, \mathcal{R}'}$, $V_{a,a'}$ and $V_{\mathcal{R},a}$ imply
\begin{eqnarray*}
\sum_{\stackrel{\mathcal{R}\in\mathrm{R}_k(z)}{a\in\mathrm{A}_k(z)}}\abs{\Delta V_{\mathcal{R},a}(z)}+\sum_{\mathcal{R},\mathcal{R'}\in\mathrm{R}_{k}(z)}\abs{\Delta V_{\mathcal{R},\mathcal{R}'}(z)}+\sum_{a,a'\in\mathrm{A}_k(z)}\abs{\Delta V_{a,a'}(z)}\\
+\sum_{\mathcal{R}\in\mathrm{R}_k(z)}\abs{\Delta V_\mathcal{R}(z)}+\sum_{a\in\mathrm{A}_k(z)}\abs{\Delta V_a(z)}\leq\f{C(\delta,\epsilon)}{(2^k)^2}.
\end{eqnarray*}
Hence, 
\begin{eqnarray*}
\sum_{\textnormal{condition (f)}}\abs{\Delta V_{\mathcal{R},a}(z)}+
\sum_{\textnormal{condition (d)}}\abs{\Delta V_{\mathcal{R},\mathcal{R}'}(z)}+\sum_{\textnormal{condition (d)}}\abs{\Delta V_{a,a'}(z)}\\ 
+\sum_{\textnormal{condition (b)}}\abs{\Delta V_\mathcal{R}(z)}+\sum_{\textnormal{condition (b)}}\abs{\Delta V_a(z)} \\
=  \sum_{\substack{k\\2^{k+1}>\im z}}\left(\sum_{\stackrel{\mathcal{R}\in\mathrm{R}_k(z)}{a\in\mathrm{A}_k(z)}} \abs{\Delta V_{\mathcal{R},a}(z)}+\sum_{\mathcal{R},\mathcal{R'}\in\mathrm{R}_{k}(z)}\abs{\Delta V_{\mathcal{R},\mathcal{R}'}(z)}+\sum_{a,a'\in\mathrm{A}_k(z)}\abs{\Delta V_{a,a'}(z)}\right)\\
+\sum_{\substack{k\\2^{k+1}>\im z}}\left(\sum_{\mathcal{R}\in\mathrm{R}_k(z)}\abs{\Delta V_\mathcal{R}(z)}+\sum_{a\in\mathrm{A}_k(z)}\abs{\Delta V_a(z)}\right)\\
 \leq \sum_{\substack{k\\2^{k+1}>\im z}}\f{C(\delta,\epsilon)}{(2^k)^2}\leq\f{C(\delta,\epsilon)}{(\im z)^2}.\\
\end{eqnarray*}
Thus, $\abs{\Delta V(z)}\leq\f{C(\delta,\epsilon)}{(\im z)^2}$.  The same logic and method of proof shows that $\abs{\partial V}\leq\f{C(\delta,\epsilon)}{\im z}$.

We now turn to demonstrating that $V$ gives rise to Carleson measures.  Fix a square $Q=Q(I)$ with $\abs{I}=d$.  Let $\mathrm{N}_k(z)$ and $\mathrm{A}_k(z)$ be as above.  Let $n$ be the integer such that $2^n\leq d <2^{n+1}$, and let 
$$
\mathrm{R}^+(z):=\bigcup_{k\geq n+4}\mathrm{R}_k(z), \quad \mathrm{R}^+:=\bigcup_{z\in Q}\mathrm{R}^+(z),
$$
$$
\mathrm{A}^+(z):=\bigcup_{k\geq n+4}\mathrm{A}_k(z), \quad \mathrm{A}^+:=\bigcup_{z\in Q}\mathrm{A}^+(z),
$$
and let
$$
\mathrm{R}^-(z):=\left(\bigcup_{k<n+4}\mathrm{R}_k(z)\right)\cup\{\mathcal{R}\in\mathfrak{R}:z\in\Gamma_{\mathcal{R}}^{\delta'}\},\quad \mathrm{R}^-:=\bigcup_{z\in Q}\mathrm{R}^-(z),
$$

$$
\mathrm{A}^-(z):=\left(\bigcup_{k<n+4}\mathrm{A}_k(z)\right)\cup\{a\in\sigma_1:z\in T_{a}^{\delta'} (\textnormal{or } T_{a, -\overline{a}}^{\delta'})\},\quad \mathrm{A}^-:=\bigcup_{z\in Q}\mathrm{A}^-(z).
$$
By construction we have that if $S_{\delta'}\cap Q\neq\emptyset$ for a slit $S$ of $\rk S=k\geq n+4$, corresponding to a component $\mathcal{R}$, then $\Gamma^{\delta'}_{\mathcal{R}}\cap Q=\emptyset$ and similarly, for all other slits $S'$ corresponding to the component, we will have $S'_{\delta'}\cap Q=\emptyset$.  Thus, $\mathrm{R}^+\cap\mathrm{R}^-=\emptyset$.  It is also possible to show that $\mathrm{A}^+\cap\mathrm{A}^-=\emptyset$ as well.

If $\Delta V_\mathcal{R}(z)\neq 0$ for a point $z\in Q$ (respectively $\Delta V_a(z)\neq 0$ for $z\in Q$) then $\mathcal{R}\in\mathrm{R}^+\cup\mathrm{R}^-$ (respectively $a\in\mathrm{A}^+\cup\mathrm{A}^-$).   Similarly, if $\Delta V_{\mathcal{R},\mathcal{R}'}(z)\neq 0$ for a point in $z\in Q$  (respectively, $\Delta V_{a,a'}(z)\neq 0$ for $z\in Q$) then $\mathcal{R}, \mathcal{R}'\in\mathrm{R}^+\cup\mathrm{R}^-$ (respectively, $a,a'\in\mathrm{A}^+\cup\mathrm{A}^-$).  Additionally, if we have that $\Delta V_{\mathcal{R},a}(z)\neq 0$ for $z\in Q$ then we must have $\mathcal{R}\in\mathrm{R}_+\cup\mathrm{R}_-$ and $a\in\mathrm{A}_+\cup\mathrm{A}_-$.

By the same logic as above, we can conclude that
\begin{eqnarray*}
\sum_{\mathcal{R}\in\mathrm{R}^+,a\in\mathrm{A}^+}\abs{\Delta V_{\mathcal{R},a}(z)}+\sum_{\mathcal{R},\mathcal{R'}\in\mathrm{R}^+}\abs{\Delta V_{\mathcal{R},\mathcal{R}'}(z)}+\sum_{a,a'\in\mathrm{A}^+}\abs{\Delta V_{a,a'}(z)}\\
+\sum_{\mathcal{R}\in\mathrm{R}^+}\abs{\Delta V_\mathcal{R}(z)}+\sum_{a\in\mathrm{A}^+}\abs{\Delta V_a(z)}\leq\f{C(\delta, \epsilon)}{d^2}=\f{C(\delta, \epsilon)}{\abs{I}^2},
\end{eqnarray*}
and so
\begin{eqnarray*}
\iint_Q\left(\sum_{\mathcal{R}\in\mathrm{R}^+,a\in\mathrm{A}^+}\abs{\Delta V_{\mathcal{R},a}(z)}+\sum_{\mathcal{R},\mathcal{R'}\in\mathrm{R}^+}\abs{\Delta V_{\mathcal{R},\mathcal{R}'}(z)}+\sum_{a,a'\in\mathrm{A}^+}\abs{\Delta V_{a,a'}(z)}\right)\im z\ dxdy\\
+\iint_Q\left(\sum_{\mathcal{R}\in\mathrm{R}^+}\abs{\Delta V_\mathcal{R}(z)}+\sum_{a\in\mathrm{A}^+}\abs{\Delta V_a(z)}\right)\im z\ dxdy \leq C(\delta, \epsilon)d=C(\delta,\epsilon)\abs{I}.
\end{eqnarray*}
By the geometry of the construction, for any $\mathcal{R}\in\mathrm{R}^-$ and the estimates on the Laplacian of $V_\mathcal{R}$, we have
$$
\iint_Q\abs{\Delta V_\mathcal{R}(z)}\im z\ dxdy\leq C(\delta,\epsilon)l(\p\mathcal{R}\cup \p(-\overline{\mathcal{R}})\cap Q(20I))\leq C(\delta,\epsilon)l(\p\mathcal{R}\cap Q(20I))
$$
with $l$ the arc length.  For $a\in\mathrm{A}^-$, $a\in Q(20I)$, and the estimates on $V_a$ imply
$$
\iint_Q\abs{\Delta V_a(z)}\im z\ dxdy\leq C(\delta,\epsilon)\im a.
$$
Similarly, by the construction for we have that for $\mathcal{R},\mathcal{R}'\in\mathrm{R}^-$ and the estimates we have on $V_{\mathcal{R},\mathcal{R}'}$ we have
\begin{eqnarray*}
\iint_Q\abs{\Delta V_{\mathcal{R},\mathcal{R}'}(z)}\im z\ dxdy & \leq & C(\delta,\epsilon)l(\p\mathcal{R}\cup \p\mathcal{R}'\cap Q(20I))\\
 & \leq & C(\delta,\epsilon)\left(l(\p\mathcal{R}\cap Q(20I))+l(\p\mathcal{R}'\cap Q(20I))\right)
\end{eqnarray*}
While for $a, a'\in\mathrm{A}^-$, $a,a'\in Q(20I)$, and the estimates on $V_{a,a'}$ imply
$$
\iint_Q\abs{\Delta V_{a,a'}(z)}\im z\ dxdy\leq C(\delta,\epsilon)\left(\im a +\im a'\right).
$$
Using the above estimates, and the fact that $\mu=\sum_{\mathcal{R}\in\mathfrak{R}}l_{\mathcal{R}}+\sum_{a\in\sigma_1}\im a\, \delta_a$ is a Carleson measure, we can conclude
\begin{eqnarray*}
\iint_Q\left(\sum_{\mathcal{R},\mathcal{R'}\in\mathrm{R}^-}\abs{\Delta V_{\mathcal{R},\mathcal{R}'}(z)}+\sum_{a,a'\in\mathrm{A}^-}\abs{\Delta V_{a,a'}(z)}+\sum_{\mathcal{R}\in\mathrm{R}^-}\abs{\Delta V_\mathcal{R}(z)}+\sum_{a\in\mathrm{A}^-}\abs{\Delta V_a(z)}\right)\im z\ dxdy & & \\
\leq C(\delta,\epsilon)d=C(\delta,\epsilon)\abs{I}.
\end{eqnarray*}

If we have $\mathcal{R}\in\mathrm{R}^+$ and $\mathcal{R}'\in\mathrm{R}^-$ then we combine these two previous estimates and the construction to have
$$
\iint_Q\abs{\Delta V_{\mathcal{R},\mathcal{R}'}(z)}\im z\ dxdy \leq C(\delta,\epsilon)\left(\abs{I}+\left(l(\p\mathcal{R}'\cap Q(20I)\right)\right).
$$
To see this we note that the support of $\Delta V_{\mathcal{R},\mathcal{R}'}\cap Q(I)$ is
$$
Q\cap\left(\Gamma^{\delta'}_{\mathcal{R}}\cup\left(\bigcup_{S\in\mathcal{S}(\mathcal{R})}S_{\delta'}\right)\right)\cap\left(\Gamma^{\delta'}_{\mathcal{R}'}\cup\left(\bigcup_{S\in\mathcal{S}(\mathcal{R}')}S_{\delta'}\right)\right).
$$
On the set $Q\cap\left(\Gamma^{\delta'}_{\mathcal{R}}\cup\left(\bigcup_{S\in\mathcal{S}(\mathcal{R})}S_{\delta'}\right)\right)$ we use the estimates of $\Delta V_{\mathcal{R},\mathcal{R}'}(z)$ as above, and on the other set $Q\cap\left(\Gamma^{\delta'}_{\mathcal{R}'}\cup\left(\bigcup_{S\in\mathcal{S}(\mathcal{R}')}S_{\delta'}\right)\right)$ we use the properties of the construction of the region $\mathcal{R}'$.
And so we find
$$
\iint_Q\sum_{\mathcal{R}\in\mathrm{R}^+,\mathcal{R}'\in\mathrm{R}^-}\abs{\Delta V_{\mathcal{R},\mathcal{R}'}(z)}\im z\ dxdy \leq C(\delta,\epsilon)\abs{I}.
$$
In the case when we have $a\in\mathrm{A}^+$ and $a'\in\mathrm{A}^-$ with $a,a'\in Q(20I)$, a similar idea allows for the estimates for the function $V_{a,a'}$ to imply
$$
\iint_Q\abs{\Delta V_{a,a'}(z)}\im z\ dxdy\leq C(\delta,\epsilon)\left(\abs{I}+\im a'\right).
$$
and so
$$
\iint_Q\sum_{a\in\mathrm{A}^+,a'\in\mathrm{A}^-}\abs{\Delta V_{a,a'}(z)}\im z\ dxdy\leq C(\delta,\epsilon)\abs{I}.
$$
The final case is when we have $\mathcal{R}\in\mathrm{R}^+$ and $a\in\mathrm{A}^-$.  Here, in this case the ideas above and the support conditions and estimates on $\Delta V_{\mathcal{R},a}$ imply that
$$
\iint_Q\sum_{\mathcal{R}\in\mathrm{R}^+,a\in\mathrm{A}^-}\abs{\Delta V_{\mathcal{R},a}(z)}\im z\ dxdy\leq C(\delta,\epsilon)\abs{I}.
$$

These estimates all together imply that $\abs{\Delta V(z)}\im z\ dxdy$ is a Carleson measure.  An identical argument implies that $\abs{\p V(z)}\,dxdy$ is a Carleson measure with intensity $C(\delta,\epsilon)$.

To complete the verifications of the properties of $V$, we must show that $V$ is sufficiently close to an appropriate branch of $\log f_1$ on the set $\{z:\abs{f_2(z)}<\delta'\}$ where $0<\delta'$ is small compared to $\delta$ and $\epsilon$.
To accomplish this, we need the following proposition, which is a consequence of Hall's Lemma, see \cite{Duren} or \cite{Garnett}.  Let $f\in H^\infty(\C_+)$ with $\norm{f}_\infty\leq 1$ and let $Q$ be a square.  Let $\eta$ be a constant $0<\eta<1$.  Let 
$$
E_\eta:=\{z\in Q:\abs{f(z)}<\eta\}.
$$
The sets $E_\eta^{\re}$ and $E_\eta^{\im}$ will denote the vertical and horizontal projections of the set $E_\eta$.
\begin{lm}
\label{HallLemma}
For a given $0<\eta<1$ there exists a constant $\gamma=\gamma(\eta)$, $0<\gamma<\eta$ such that for any $f\in H^\infty(\C_+)$ with $\norm{f}_\infty\leq1$ satisfying
$$
\max\left\{\abs{E_\epsilon^{\re}},\abs{E_\epsilon^{\im}}\right\}\geq\f{1}{4}\abs{I}
$$
the inequality $\abs{f(z)}<\eta$ holds for all $z$ in the upper half of $Q$.
\end{lm}

Let $\mathcal{O}$ be a connected component of the set $\{z\in\C_+:\abs{q(z)}<\eta\}$ where $\gamma=\gamma(\eta)$ from the above lemma.  By the maximum modulus principle, the set $\mathcal{O}$ is simply connected.  Let $n\in\Z$ be the smallest integer such that there exists a square $Q=Q(I)$ with base $I$, $\abs{I}=2^n$ for which $\mathcal{O}\subset Q$.

\begin{lm}
\label{lm7}
If, for a slit S, we have $S_{\delta'}\cap\mathcal{O}\neq\emptyset$ then the rank of S is at least $n-3$.
\end{lm}
\begin{proof}
Let $\mathcal{O}^{\re}$ and $\mathcal{O}^{\im}$ be the vertical and horizontal projections of $\mathcal{O}$.  If
$$
\max\{\abs{\mathcal{O}^{\re}},\abs{\mathcal{O}^{\im}}\}<\f{1}{4}\abs{I}=2^{n-2}
$$
then by the definition of $n$ and the fact that $\mathcal{O}$ is simply connected, we have for all $z\in\mathcal{O}$, $\im z>\f{1}{4}\abs{I}$.  Hence, if $\rk S<n-3$ then $S_{\delta'}\cap\mathcal{O}=\emptyset$.

Now suppose that
$$
\max\{\abs{\mathcal{O}^{\re}},\abs{\mathcal{O}^{\im}}\}\geq\f{1}{4}\abs{I}=2^{n-2}.
$$
Let $S$ be a slit of rank k, with $k<n-3$, and $S_{\delta'}\cap\mathcal{O}\neq\emptyset$.  Let $z_0$ be the origin of the slit $S$, i.e., either a point in $\p\mathcal{R}_{\delta'}$ or $\sigma_1$.  Let $J$ be the interval in $\R$ with center at the point $\re z_0$ or length $2\im z_0$.  Since $\mathcal{O}$ is connected, for the set $E:=\mathcal{O}\cap Q(J)$ we have
$$
\max\{\abs{E^{\re}},\abs{E^{\im}}\}>\f{1}{4}\abs{J}.
$$
By Proposition \ref{HallLemma}, $\abs{q(z_0)}<\delta'$.  This leads to a contradiction, since we were working under the assumption that $p$ and $q$ were Corona Data, namely, 
$$
\inf_{z\in\C_+}\left(\abs{p(z)}+\abs{q(z)}\right)=3\delta'>0
$$
and by construction of the Carleson regions, $\abs{p(z_0)}<\delta'$.
\end{proof}

Fix a point $z_0\in\mathcal{O}$.  For each connected component $\mathcal{R}$ with $\mathcal{R}\neq-\overline{\mathcal{R}}$, define on the set $\mathcal{O}$ a branch of $\log(\alpha_\mathcal{R}B_{\mathcal{R}}(z)B_{-\overline{\mathcal{R}}}(z))$, $\alpha_R\in\C$ a unimodular number, such that $\log(\alpha_\mathcal{R}B_\mathcal{R}(z_0)B_{-\overline{\mathcal{R}}}(z_0))=V_\mathcal{R}(z_0)$.  

If, for a component $\mathcal{R}$, there is a slit corresponding to it such that $S_{\delta'}\cap\mathcal{O}\neq\emptyset$, then
$$
\abs{V_\mathcal{R}(z)-\log(\alpha_\mathcal{R}B_\mathcal{R}(z)B_{-\overline{\mathcal{R}}}(z))}\leq NC(\delta,\epsilon),
$$
where $N$ is the number of slits $S$ corresponding to $\mathcal{R}$ for which $S_{\delta'}\cap\mathcal{O}\neq\emptyset$.  As above, we let $n$ be the smallest integer for which $\mathcal{O}\subset Q$, with $Q=Q(I)$ such a square.  Let $\mathfrak{R}_0$ be the set of all components $\mathcal{R}\in\mathfrak{R}$ with $\mathcal{R}\neq-\overline{\mathcal{R}}$ such that there exist no slits, $\rk S\geq n+4$ corresponding to this component and satisfying $S_{\delta'}\cap\mathcal{O}\neq\emptyset$.  

By Lemma \ref{lm2}, for any $k\geq n-3$ there exists at most $C(\delta,\epsilon)$ slits of rank $k$ such that $S_{\delta'}\cap\mathcal{O}\neq\emptyset$.  By Lemma \ref{lm7}, there exists no slits $S$ with $\rk S<n-3$ for which $S_{\delta'}\cap\mathcal{O}\neq\emptyset$, and therefore at most $C(\delta,\epsilon)$ slits $S$, with $\rk S<n+4$ satisfying $S_{\delta'}\cap\mathcal{O}\neq\emptyset$.  It then follows that
$$
\sum_{\stackrel{\mathcal{R}\in\mathfrak{R}_0}{\mathcal{R}\neq-\overline{\mathcal{R}}}}\abs{V_\mathcal{R}(z)-\log(\alpha_\mathcal{R}B_\mathcal{R}(z)B_{-\overline{\mathcal{R}}}(z))}\leq C(\delta,\epsilon),\quad z\in\mathcal{O}.
$$

For a component $\mathcal{R}$ with $\mathcal{R}\neq-\overline{\mathcal{R}}$ but not in $\mathfrak{R}_0$, it is possible to obtain a better estimate.  Indeed, if for a component $\mathcal{R}$ there exists a slit $S$ of rank $k\geq n+4$ satisfying $S_{\delta'}\cap\mathcal{O}\neq\emptyset$, then for other slits $S'$ corresponding to this component, $S'_{\delta'}\cap\mathcal{O}=\emptyset$.  Therefore, 
\begin{eqnarray*}
\abs{V_\mathcal{R}(z)-\log(\alpha_\mathcal{R}B_\mathcal{R}(z)B_{-\overline{\mathcal{R}}}(z))} & \leq & \left(\sup_{z\in S_{\delta'}}\abs{V_\mathcal{R}'(z)}+\sup_{z\in S_{\delta'}}\abs{(B_\mathcal{R}(z)B_{-\overline{\mathcal{R}}}(z))'}\right)\textnormal{diam}(S_{\delta'}\cap\mathcal{O})\\
 & \leq & C(\delta,\epsilon)2^{-k}\textnormal{diam}\,Q\leq C(\delta,\epsilon) 2^{n-k}.
\end{eqnarray*}
We sum this estimate over $k\geq n+4$ to obtain
$$
\left\vert\sum_{\stackrel{\mathcal{R}\notin\mathfrak{R}_0}{\mathcal{R}\neq-\overline{\mathcal{R}}}}V_\mathcal{R}(z)-\log(\alpha_\mathcal{R}B_\mathcal{R}(z)B_{-\overline{\mathcal{R}}}(z))\right\vert\leq C(\delta,\epsilon), \quad z\in\mathcal{O}
$$
Coupling this estimate with the one from above we can conclude 
$$
\left\vert\sum_{\stackrel{\mathcal{R}\in\mathfrak{R}}{\mathcal{R}\neq -\overline{\mathcal{R}}}}V_\mathcal{R}(z)-\log(\alpha_\mathcal{R}B_\mathcal{R}(z)B_{-\overline{\mathcal{R}}}(z))\right\vert\leq C(\delta,\epsilon), \quad z\in\mathcal{O}.
$$

Analogously, for any $a\in\sigma_1\cap \left(\C_+\setminus i\R\right)$, applying identical reasoning lets us conclude that there exists a constant $\alpha_a$ with $|\alpha_a|=1$ and a branch of $\log(\alpha_a b_a(z)b_{-\overline{a}}(z))$ so that $|V_a-\log(\alpha_a b_a(z)b_{-\overline{a}}(z)) |\leq C$ for $z\in\mathcal{O}$.

We next consider the case of $V_{\mathcal{R},\mathcal{R}'}$.  Fix a point $z_0\in\mathcal{O}$.  For each connected component $\mathcal{R}\cup\mathcal{R}\cup I_{\mathcal{R},\mathcal{R}'}$, define on the set $\mathcal{O}$ a branch of $\log(\alpha_\mathcal{R}B_{\mathcal{R}}(z)B_{\mathcal{R'}}(z))$, $\alpha_R\in\C$ a unimodular number, such that $\log(\alpha_\mathcal{R}B_\mathcal{R}(z_0)B_{\mathcal{R}'}(z_0))=V_\mathcal{R}(z_0)$.  

If, for a component $\mathcal{R}\cup\mathcal{R}'\cup I_{\mathcal{R},\mathcal{R}'}$, there is a slit corresponding to it such that $S_{\delta'}\cap\mathcal{O}\neq\emptyset$, then
$$
\abs{V_{\mathcal{R},\mathcal{R}'}(z)-\log(\alpha_\mathcal{R}B_\mathcal{R}(z)B_{\mathcal{R}'}(z))}\leq NC(\delta,\epsilon),
$$
where $N$ is the number of slits $S$ corresponding to $\mathcal{R}$ for which $S_{\delta'}\cap\mathcal{O}\neq\emptyset$.  Again let $n$ be the smallest integer for which $\mathcal{O}\subset Q$, with $Q=Q(I)$ such a square.  Let $\mathfrak{R}_0$ be the set of all components $\mathcal{R}\cup\mathcal{R}\cup I_{\mathcal{R},\mathcal{R}'}$ such that there exist no slits, $\rk S\geq n+4$ corresponding to this component and satisfying $S_{\delta'}\cap\mathcal{O}\neq\emptyset$.  

By Lemma \ref{lm2}, for any $k\geq n-3$ there exists at most $C(\delta,\epsilon)$ slits of rank $k$ such that $S_{\delta'}\cap\mathcal{O}\neq\emptyset$.  By Lemma \ref{lm7}, there exists no slits $S$ with $\rk S<n-3$ for which $S_{\delta'}\cap\mathcal{O}\neq\emptyset$, and therefore at most $C(\delta,\epsilon)$ slits $S$, with $\rk S<n+4$ satisfying $S_{\delta'}\cap\mathcal{O}\neq\emptyset$.  It then follows that
$$
\sum_{\mathcal{R}\cup\mathcal{R}'\cup I_{\mathcal{R},\mathcal{R}'}\in\mathfrak{R}_0}\abs{V_{\mathcal{R},\mathcal{R}'}(z)-\log(\alpha_\mathcal{R}B_\mathcal{R}(z)B_{\mathcal{R}'}(z))}\leq C(\delta,\epsilon),\quad z\in\mathcal{O}.
$$

For a component $\mathcal{R}\cup\mathcal{R}'\cup I_{\mathcal{R},\mathcal{R}'}\notin\mathfrak{R}_0$, it is possible to obtain a better estimate.  Indeed, if for a component $\mathcal{R}\cup\mathcal{R}'\cup I_{\mathcal{R},\mathcal{R}'}$ there exists a slit $S$ of rank $k\geq n+4$ satisfying $S_{\delta'}\cap\mathcal{O}\neq\emptyset$, then for other slits $S'$ corresponding to this component, $S'_{\delta'}\cap\mathcal{O}=\emptyset$.  Therefore, 
\begin{eqnarray*}
\abs{V_{\mathcal{R},\mathcal{R}'}(z)-\log(\alpha_\mathcal{R}B_\mathcal{R}(z)B_{\mathcal{R}'}(z))} & \leq & \left(\sup_{z\in S_{\delta'}}\abs{V_{\mathcal{R},\mathcal{R}'}'(z)}+\sup_{z\in S_{\delta'}}\abs{(B_\mathcal{R}(z)B_{\mathcal{R}'}(z))'}\right)\textnormal{diam}(S_{\delta'}\cap\mathcal{O})\\
 & \leq & C(\delta,\epsilon)2^{-k}\textnormal{diam}\,Q\leq C(\delta,\epsilon) 2^{n-k}.
\end{eqnarray*}
We sum this estimate over $k\geq n+4$ and can conclude 
$$
\left\vert\sum_{\mathcal{R}\cup\mathcal{R}'\cup I_{\mathcal{R},\mathcal{R}'}\notin\mathfrak{R}_0}V_{\mathcal{R},\mathcal{R}'}(z)-\log(\alpha_\mathcal{R}B_\mathcal{R}(z)B_{\mathcal{R}'}(z))\right\vert\leq C(\delta,\epsilon), \quad z\in\mathcal{O}
$$
Coupling this estimate with the one from above we can conclude 
$$
\left\vert\sum_{\stackrel{\mathcal{R},\mathcal{R}'\in\mathfrak{R}}{\mathcal{R}=-\overline{\mathcal{R}}, \mathcal{R}'=-\overline{\mathcal{R}'}}}V_{\mathcal{R},\mathcal{R}'}(z)-\log(\alpha_\mathcal{R}B_\mathcal{R}(z)B_{\mathcal{R}'}(z))\right\vert\leq C(\delta,\epsilon), \quad z\in\mathcal{O}.
$$


The case of $V_{a,a'}$ is analogous to the case of $V_a$ and so we omit it.

Combining all these estimates together then, in turn, implies that the function $V$ is sufficiently close to an appropriate branch of logarithm of $p$ and hence $V$ has all the necessary properties.  This then proves Proposition \ref{CorrectingV_new}.

\end{document}